\RequirePackage{fix-cm}
\documentclass[smallextended, envcountsame]{svjour3}
\smartqed
\journalname{}

\usepackage[unicode = false,
    pdftoolbar = true,
    colorlinks = true,
    linkcolor = blue,
    citecolor = blue,
    filecolor = black,
    urlcolor = blue,
    breaklinks = true]{hyperref}

\usepackage[utf8]{inputenc}
\usepackage{listings}  
\usepackage{enumitem}  
\usepackage{amsmath}   
\usepackage{amssymb}   
\usepackage{xcolor}    
\usepackage{graphicx}  
\usepackage{verbatim}
\usepackage{geometry}
\usepackage{mathtools}
\usepackage{float}
\usepackage{bm}
\usepackage[ruled]{algorithm2e}
\usepackage{placeins}
\usepackage{listings}
\usepackage{multirow}
\usepackage{amsthm}
\usepackage[justification=centering]{caption}

\definecolor{gab}{HTML}{50c878}
\definecolor{pat}{HTML}{bf68ff}
\definecolor{sham}{HTML}{7faaff}

\allowdisplaybreaks[4]

\DeclareMathOperator*{\argmax}{argmax}

\DeclareMathAlphabet{\mymathbb}{U}{BOONDOX-ds}{m}{n}

\makeatletter
\newcommand*{\rom}[1]{\expandafter\@slowromancap\romannumeral #1@}
\makeatother

\title{Accuracy and Relationships of Quadratic Models in Derivative-free Optimization}

\author{Yiwen Chen \and Warren Hare \and Lindon Roberts}
\institute{Yiwen Chen \at Department of Mathematics, University of British Columbia, Kelowna, British Columbia, V1V 1V7, Canada.  Chen's research is partially funded by the Natural Sciences and Engineering Research Council of Canada (les recherches de Chen sont partiellement financ\'ees par le Conseil de recherches en sciences naturelles et en g\'enie du Canada), Discover Grant \#2023-03555.  Chen was also supported by the Australian Research Council Discovery Early Career Award DE240100006.\\
\email{yiwchen@student.ubc.ca} \and Warren Hare \at Department of Mathematics, University of British Columbia, Kelowna, British Columbia, V1V 1V7, Canada.  Hare's research is partially funded by the Natural Sciences and Engineering Research Council of Canada (les recherches de Hare sont partiellement financ\'ees par le Conseil de recherches en sciences naturelles et en g\'enie du Canada), Discover Grant \#2023-03555.\\ \email{warren.hare@ubc.ca} \and Lindon Roberts \at School of Mathematics and Statistics, University of Melbourne, Parkville, Victoria 3010, Australia.  Roberts was supported by the Australian Research Council Discovery Early Career Award DE240100006.\\
\email{lindon.roberts@unimelb.edu.au}}
\date{\today}

\begin{document}

\maketitle

\begin{abstract}
    We study three quadratic models in model-based derivative-free optimization: the minimum norm (MN), minimum Frobenius norm (MFN), and quadratic generalized simplex derivative (QS) models.  Despite their widespread use, their approximation accuracy and relationships have not been systematically explored.  We establish fully linear error bounds for all three models, removing the uniformly bounded model Hessian assumption required in existing MN analyses and deriving the first such results for the QS model.  We further analyze Hessian approximation accuracy via directional error bounds, showing that all three models achieve fully quadratic accuracy along sample directions under a mild condition on the sample set.  This reveals a form of directional fully quadratic accuracy not captured by existing theory. Finally, we characterize the relationships among these models, identifying conditions under which they coincide and clarifying their structural connections.
\end{abstract}

\vspace{0.1cm}

\vspace{0.1cm}

\section{Introduction}
    Derivative-free optimization (DFO) methods are an essential tool for solving optimization problems where the derivatives of the objective and/or constraint functions are unavailable, unreliable, or prohibitively expensive to compute \cite{audet2017derivative,conn2009introduction}.   One major class of DFO methods, known as model-based DFO methods, constructs surrogate models using only function evaluations.  The approximation accuracy of these models depends critically on the geometry of the set of sampled points, which we refer to as a \emph{sample set}. 

    Convergence analyses of model-based DFO methods typically rely on model accuracy assumptions, most notably that the models are \emph{fully linear} or \emph{fully quadratic} \cite{Audet2020,roberts2025introduction}. A common approach for constructing such models is quadratic interpolation. In this paper, we study three quadratic models: the minimum norm (MN) model \cite{conn1996algorithm}, the minimum Frobenius norm (MFN) model~\cite{conn1998derivative}, and the quadratic generalized simplex derivative (QS) model \cite{chen2026relationships}. The QS model is constructed using generalized simplex gradients \cite{conn2008geometry,custodio2008using,custodio2007using,regis2015calculus} and generalized simplex Hessians \cite{hare2024matrix}.  Although these models have appeared in various forms in the literature, a unified theoretical comparison of their approximation accuracy and structural relationships remains lacking.

    Our first contribution is the derivation of fully linear error bounds for all three models.  Existing results for the MN model rely on a uniformly bounded model Hessian, while those for the MFN model are established under different settings or sample set geometry measures.  We remove the bounded Hessian assumption for the MN model and establish the \emph{first} fully linear error bounds for the QS model, thereby providing a unified perspective on their approximation accuracy.
    
    We next study the accuracy of the Hessian approximations arising from these models through directional error bounds.  Under a mild structural assumption on the sample set, we show that all three models achieve fully quadratic accuracy along sample directions.  More precisely, the directional error ranges from $\mathcal{O}(1)$ to $\mathcal{O}(\Delta_{D^Y})$, where $\Delta_{D^Y}$ is the maximum length of the sample directions.  This reveals that these models can significantly exceed the accuracy of standard fully linear models in a directional sense.  Similar phenomena have been observed in \cite{Cao2023}, where the sharp error bound of linear interpolation/extrapolation depends on the location of the point of interest. These findings provide new insights into how second-order information can be captured without requiring fully determined quadratic interpolation.
    
    Finally, we analyze the relationships among the three models.  We identify conditions under which they coincide and clarify the structural connections between their constructions. These results not only enhance the theoretical understanding of these models but also suggest simpler and potentially more efficient construction procedures.
    
    In summary, this paper advances the theoretical foundations of model-based DFO by establishing new fully linear error bounds, deriving directional Hessian error bounds, and clarifying the relationships among several widely used quadratic models. These results have implications for both the analysis and practical implementation of DFO methods.

    The remainder of the paper is organized as follows.  Section~\ref{sec:notation&defs} introduces notation and background definitions.  Section~\ref{sec:fullylinearebs} establishes fully linear error bounds.  Section~\ref{sec:directionalHessebs} presents directional Hessian error bounds.  Section~\ref{sec:relationships} analyzes the relationships among the three models.  Section~\ref{sec:conclusion} concludes this paper and discusses future directions.

\section{Notation and Background}\label{sec:notation&defs}
    In this section, we introduce the notation and background definitions used throughout the paper. We begin by presenting the notation, followed by the formal definitions, and defer all examples to Subsection \ref{subsec:examples}.

\subsection{Notation}
    We use $\mymathbb{0}$ and $\mymathbb{1}$ to denote the all-zero and all-one vectors with appropriate dimensions.  We use $e^i$ to denote the $i$-th coordinate vector with appropriate dimensions.  We use $I_n$ to denote the identity matrix of dimension $n\times n$.  We use $\mymathbb{0}_{m\times n}$ to denote the all-zero matrix of dimension $m\times n$.  We use $S_n(\mathbb{R})$ to denote the set of all $n\times n$ real symmetric matrices.  For any matrix $A$, we use $A_{ij}$ to denote its $(i,j)$-th element and $\mathrm{col}(A)$ to denote its column space.  We use $a\in A$ to mean that $a$ is a column of $A$ and define $\Delta_A := \max_{a\in A}\|a\|$.  For a non-zero matrix $A$, we define $\overline{A}:=A\slash\Delta_A$. For a square matrix $A$, we use $\mathrm{tr}(A)$ to denote its trace and $\mathrm{diag}(A)$ to denote the vector consisting of its diagonal elements.  For a vector $v\in\mathbb{R}^n$, we use $\mathrm{Diag}(v)$ to denote the $n\times n$ diagonal matrix whose diagonal elements are given by the components of $v$.  We use $\|\cdot\|$ to denote the Euclidean norm of a vector and the spectral norm of a matrix.  We use $\|\cdot\|_F$ to denote the Frobenius norm of a matrix.  We use $B(x;\Delta)$ to denote the closed Euclidean ball centered at $x$ with radius $\Delta$.  We use $[n]$ to denote the set $\{1,\ldots,n\}$. For any set $Y$ written in the form $Y=\{x^0,x^0+d^i:i\in [m]\}$, we define $D^Y := [d^1\cdots d^m]$.

\subsection{Generalized simplex gradient, generalized simplex Hessian, and quadratic generalized simplex derivative models}
    \begin{definition}
        Let $x^0\in\mathbb{R}^n$ and $S=[s^1\cdots s^p]\in\mathbb{R}^{n\times p}$.  The generalized simplex gradient (GSG) of $f$ at $x^0$ over $S$ is $$\nabla_sf(x^0;S) := \left(S^\top\right)^\dagger\delta_f(x^0;S)\in\mathbb{R}^n$$ where the $i$-th element of $\delta_f(x^0;S)$ is $f(x^0 + s^i) - f(x^0)$, $i\in [p]$.
    \end{definition}
    The sample set of the GSG of $f$ at $x^0$ over $S$ is $Y=\{x^0,x^0+s^i:i\in [p]\}$.
    
    \begin{definition}
        Let $x^0\in\mathbb{R}^n$ and $S=[s^1\cdots s^p]\in\mathbb{R}^{n\times p}$.  Let matrices $T_i=[t_i^1\cdots t_i^{q_i}]\in\mathbb{R}^{n\times q_i}$, $i\in [p]$.  Denote $T_{1:p}=\{T_1,\ldots,T_p\}$.  The generalized simplex Hessian (GSH) of $f$ at $x^0$ over $S$ and $T_{1:p}$ is $$\nabla_s^2f(x^0;S;T_{1:p}) := \left(S^\top\right)^\dagger\delta_f^2(x^0;S;T_{1:p})\in\mathbb{R}^{n\times n}$$ where the $i$-th row of $\delta_f^2(x^0;S;T_{1:p})$ is $(\nabla_sf(x^0+s^i;T_i)-\nabla_sf(x^0;T_i))^\top$, $i\in [p]$.
    \end{definition}
    \begin{remark}
        Following \cite{hare2024matrix}, if $T_1=\cdots=T_p=T\in\mathbb{R}^{n\times q}$, then, with a little misuse of notation, we denote $T_{1:p}=T$.  In this case, we also denote by $\delta_{\delta_f}(x^0;S;T)$ the matrix with the $i$-th row given by $(\delta_f(x^0+s^i;T) - \delta_f(x^0;T))^\top$.
        Equivalently, for $i\in [p]$ and $j\in [q]$,
        \begin{equation*}
            \left(\delta_{\delta_f}(x^0;S;T)\right)_{ij} := f(x^0+s^i+t^j)-f(x^0+s^i)-f(x^0+t^j)+f(x^0).
        \end{equation*}
        Notice that $\nabla_s^2f(x^0;S;T) = \left(S^\top\right)^\dagger\delta_{\delta_f}(x^0;S;T)T^\dagger$.
    \end{remark}
    \begin{remark}\label{rem:GSH(S,T)^T=GSH(T,S)}
        As we will see in Example~\ref{exp:asymmetricGSH}, the GSH is not necessarily symmetric.  In general, we have $\nabla_s^2f(x^0;S;T)^\top=\nabla_s^2f(x^0;T;S)$ \cite[Prop.~4.4]{hare2024matrix}.
    \end{remark}

    The sample set of the GSH of $f$ at $x^0$ over $S$ and $T_{1:p}$ is
    \begin{align}\label{sampleset_GSH}
    \begin{split}
        Y=\left\{x^0,x^0+s^i,x^0+t_i^j,x^0+s^i+t_i^j:i\in [p],j\in [q_i]\right\}.
    \end{split}
    \end{align}

    There is no unique way to construct quadratic models using the GSG and GSH.  In this paper, we use the definition given in \cite[Def.~7]{chen2026relationships}, which defines a general class of such models called the class of quadratic generalized simplex derivative models.
    \begin{definition}
        Let $x^0\in\mathbb{R}^n$, $Y$ be given by \eqref{sampleset_GSH}, and $d=x-x^0$.  The class of quadratic generalized simplex derivative models of $f$ at $x^0$ over $Y$ is
        \begin{align*}
            \left\{m_{\mathrm{QS}}^{x^0,Y}(x) = f(x^0)+\nabla_{\mathrm{QS}}f(x^0;Y)^\top d +\frac{1}{2}d^\top\nabla^2_{\mathrm{QS}}f(x^0;Y)d\right\}
        \end{align*}
        where $\nabla_{\mathrm{QS}}f(x^0;Y)$ and $\nabla^2_{\mathrm{QS}}f(x^0;Y)$ are linear combinations of GSG and GSH constructed using points in $Y$, chosen such that the union of the points used in both constructions equals $Y$.
    \end{definition}
    \begin{remark}
        As mentioned in Remark \ref{rem:GSH(S,T)^T=GSH(T,S)}, the GSH need not be symmetric, so the QS model is not necessarily an interpolation model.  One case where it is an interpolation model is given in Example \ref{exp:symdeltadeltafMinv_specialT} and analyzed in \cite[Sec.~5]{hare2024matrix}.  Further details and additional cases are discussed in Section \ref{sec:relationships}.
    \end{remark}

\subsection{Minimum norm models and minimum Frobenius norm models}
    Two common choices of quadratic interpolation models are the minimum norm models \cite{conn1996algorithm} and the minimum Frobenius norm models \cite{conn1998derivative}.  Here, we adopt the definition given in \cite[Defs.~2 and 3]{chen2026relationships}.

     \begin{definition}
        Let $x^0\in\mathbb{R}^n$, $Y=\{x^0,x^0+d^i:i\in [m]\}$, and $d=x-x^0$.  The class of minimum norm models of $f$ at $x^0$ over $Y$ is
        \begin{equation*}
            \left\{m_{\mathrm{MN}}^{x^0,Y}(x) = f(x^0) + \nabla_{\mathrm{MN}}f(x^0;Y)^\top d + \frac{1}{2}d^\top \nabla_{\mathrm{MN}}^2f(x^0;Y)d\right\}
        \end{equation*}
        where $(\nabla_{\mathrm{MN}}f(x^0;Y),\nabla_{\mathrm{MN}}^2f(x^0;Y))\in\mathcal{M}^{x^0,Y}_{\mathrm{MN}}$, the set of minimizers of
        \begin{align}
            \begin{split}\label{pro:MN}
                \min_{\substack{\alpha\in\mathbb{R}^n\\ H\in S_n(\mathbb{R})}}~~~&\frac{1}{2}\left\|\alpha\right\|^2+\frac{1}{2}\left\|H\right\|_F^2\\
                s.t.~~~&\left(d^i\right)^\top\alpha+\frac{1}{2}\left(d^i\right)^\top H\left(d^i\right) = f(x^0+d^i)-f(x^0),~~~i \in [m].
            \end{split}\tag{MN}
        \end{align}
    \end{definition}
    We say $Y$ is poised for MN interpolation if the solution to Problem~\eqref{pro:MN} exists and is unique.  Notice that the objective function is strongly convex and all constraints are linear with respect to $\alpha$ and $H$, the solution to Problem~\eqref{pro:MN} exists and is unique if and only if Problem~\eqref{pro:MN} is feasible, i.e., there exists a quadratic function that interpolates $f$ on $Y$.   

    \begin{definition}\label{def:MFN}
        Let $x^0\in\mathbb{R}^n$, $Y=\{x^0,x^0+d^i:i\in [m]\}$, and $d=x-x^0$.  The class of minimum Frobenius norm models of $f$ at $x^0$ over $Y$~is
        \begin{equation*}
            \left\{m_{\mathrm{MFN}}^{x^0,Y}(x) = f(x^0) + \nabla_{\mathrm{MFN}}f(x^0;Y)^\top d + \frac{1}{2}d^\top\nabla_{\mathrm{MFN}}^2f(x^0;Y)d\right\}
        \end{equation*}
        where $(\nabla_{\mathrm{MFN}}f(x^0;Y),\nabla_{\mathrm{MFN}}^2f(x^0;Y))\in\mathcal{M}^{x^0,Y}_{\mathrm{MFN}}$, the set of minimizers of
        \begin{align}
            \begin{split}\label{pro:MFN}
                \min_{\substack{\alpha\in\mathbb{R}^n\\ H\in S_n(\mathbb{R})}}~~~&\frac{1}{2}\left\|H\right\|_F^2\\
                s.t.~~~&\left(d^i\right)^\top\alpha+\frac{1}{2}\left(d^i\right)^\top H\left(d^i\right) = f(x^0+d^i)-f(x^0),~~~i\in [m].
            \end{split}\tag{MFN}
        \end{align}
    \end{definition}
    We say $Y$ is poised for MFN interpolation if the solution to Problem~\eqref{pro:MFN} exists and is unique.  An equivalent condition for the existence and uniqueness of Problem~\eqref{pro:MFN} can be found in \cite[Sec.~5.3]{conn2009introduction}.  As mentioned in \cite[Sec.~5.3]{conn2009introduction}, if $Y$ is poised for MFN interpolation, then it is poised for MN interpolation.

\subsection{Examples}\label{subsec:examples}
    In this subsection, we present examples of the MN, MFN, and QS models.  We begin with examples of the MN and MFN models. The following example shows that not all sets poised for MN interpolation are poised for MFN interpolation.
    \begin{example}\label{exp:poisedMNnotpoisedMFN}
        Consider $n=3$, $x^0=\mymathbb{0}$, $Y=\{\mymathbb{0},e^1,e^2,-e^1, -e^2\}\subseteq\mathbb{R}^3$, and $f(x)=\|x\|^2$.  We can compute that
        \begin{equation*}
            \mathcal{M}^{x^0,Y}_{\mathrm{MN}} = \left\{\left(\mymathbb{0}, \begin{bmatrix}2 & 0 & 0\\ 0 & 2 & 0\\ 0 & 0 & 0\end{bmatrix}\right)\right\},~~\mathcal{M}^{x^0,Y}_{\mathrm{MFN}} = \left\{\left(\begin{bmatrix}0\\ 0\\ r\end{bmatrix}, \begin{bmatrix}2 & 0 & 0\\ 0 & 2 & 0\\ 0 & 0 & 0\end{bmatrix}\right):r\in\mathbb{R}\right\}.
        \end{equation*}
        This implies that $Y$ is poised for MN interpolation but not poised for MFN interpolation, as the solution to Problem~\eqref{pro:MFN} is not unique.
    \end{example}

    We also notice that even if $Y$ is poised for both MN and MFN interpolation, the MN and MFN models may not be the same.
    \begin{example}\label{exp:MNmoreaccurate}
        Consider $n=2$, $x^0=\mymathbb{0}$, $Y=\{\mymathbb{0},e^1,e^2,2e^1,e^1+e^2\}\subseteq\mathbb{R}^2$, and $f(x)=\|x\|^2$.  We can compute that
        \begin{equation*}
            \mathcal{M}^{x^0,Y}_{\mathrm{MN}} = \left\{\left(\begin{bmatrix}
                0\\
                \frac{4}{5}
            \end{bmatrix}, \begin{bmatrix}
                2 & 0\\
                0 & \frac{2}{5}
            \end{bmatrix}\right)\right\}~~~\text{and}~~~\mathcal{M}^{x^0,Y}_{\mathrm{MFN}} = \left\{\left(\begin{bmatrix}
                0\\
                1
            \end{bmatrix}, \begin{bmatrix}
                2 & 0\\
                0 & 0
            \end{bmatrix}\right)\right\}.
        \end{equation*}
    \end{example}

    Next, we provide an example of the QS model illustrating that it may or may not coincide with the MN and MFN models.
    \begin{example}
        We use the same setting as Example \ref{exp:poisedMNnotpoisedMFN}.  One possible QS model can be constructed by letting
        \begin{align*}
            \nabla_{\mathrm{QS}}f(x^0;Y) &= \frac{1}{2}\left(\nabla_sf(x^0;[e^1~e^2] + \nabla_sf(x^0;[-e^1~-e^2])\right) = \mymathbb{0}\\
            \nabla^2_{\mathrm{QS}}f(x^0;Y) &= \nabla_s^2f(x^0;[e^1~e^2];\{[-e^1],[-e^2]\}) = \begin{bmatrix}2 & 0 & 0\\ 0 & 2 & 0\\ 0 & 0 & 0\end{bmatrix}.
        \end{align*}
        As seen in Example \ref{exp:poisedMNnotpoisedMFN}, this QS model coincides with the MN model and belongs to the class of MFN models of $f$ at $x^0$ over $Y$.  However, if we change $\nabla_{\mathrm{QS}}f(x^0;Y)$ to $\nabla_sf(x^0;[e^1])=e^1$, then the corresponding QS model no longer coincides with MN or MFN models.
    \end{example}

    Finally, we note that, unlike the MN and MFN models, the Hessian of the QS model need not be symmetric.  This asymmetry arises because the GSH itself may be non-symmetric, as illustrated in the following example.
    \begin{example}\label{exp:asymmetricGSH}
        Consider $n=2$, $x^0=\mymathbb{0}$, $S=[e^1~e^2]$, $T=[e^1]$, and $f(x)=(\mymathbb{1}^\top x)^2$.  Then, the GSH of $f$ at $x^0$ over $S$ and $T$ is
        \begin{equation*}
            \nabla_s^2f(x^0;S;T) = \begin{bmatrix}
                f(2e^1)-f(e^1)-f(e^1)+f(\mymathbb{0})\\
                f(e^1+e^2)-f(e^2)-f(e^1)+f(\mymathbb{0})\\
            \end{bmatrix}\left(e^1\right)^\top = \begin{bmatrix}
                2 & 0\\
                2 & 0
            \end{bmatrix}\notin S_2(\mathbb{R}).
        \end{equation*}
    \end{example}

    The examples in this subsection illustrate both distinctions and potential connections among the MN, MFN, and QS models. In this paper, we establish error bounds and further investigate the relationships among these models.

\section{Fully linear error bounds}\label{sec:fullylinearebs}
    In this section, we present and derive fully linear error bounds for the MN, MFN, and QS models.  We note that fully linear bounds of the MN and MFN models have previously been studied in \cite[Chap.~5]{conn2009introduction}, and the errors of the MN model projected onto an appropriate linear subspace are given in \cite[Thm.~5.12]{conn2008geometry}.  However, the fully linear bounds for the MN model provided in \cite[Chap.~5]{conn2009introduction} rely on the assumption that the model Hessians are uniformly bounded.  Our analysis demonstrates that this assumption is unnecessary.  Moreover, we establish the first fully linear error bounds for the QS model.

    Following \cite[Def.~9.1]{audet2017derivative} and \cite[Def.~6.1]{conn2009introduction}, we first define a class of fully linear models.
    \begin{definition}
        Suppose that $f\in\mathcal{C}^1$, $x^0\in\mathbb{R}^n$, and $\overline{\Delta}>0$.  We say that $\mathcal{M}_{\overline{\Delta}}=\{m_{\Delta}:\mathbb{R}^n\to\mathbb{R}\}_{\Delta\in (0,\overline{\Delta}]}$ is a class of fully linear models of $f$ at $x^0$ parameterized by $\Delta$ if there exists constants $\kappa^{ef}>0$ and $\kappa^{eg}>0$ independent of $x^0$ such that for all $\Delta\in (0,\overline{\Delta}]$ and $x\in B(x^0;\Delta)$,
        \begin{equation*}
            \left|f(x)-m_{\Delta}(x)\right| \le \kappa^{ef}\Delta^2~~~\text{and}~~~\left\|\nabla f(x)-\nabla m_{\Delta}(x)\right\| \le \kappa^{eg}\Delta.
        \end{equation*}
    \end{definition}

    In the following results, we may use the following assumptions.
    \begin{assumption}\label{ass:D^Y_full_rank}
        The sample set $Y=\{x^0,x^0+d^i:i\in [m]\}\subseteq \mathbb{R}^n$ satisfies that $D^Y$ has full row rank, i.e., $\mathrm{col}(D^Y)=\mathbb{R}^n$.
    \end{assumption}
    \begin{assumption}\label{ass:C1+}
        The objective $f\in\mathcal{C}^{1+}$ with $L_{\nabla f}$-Lipschitz gradient.
    \end{assumption}

    We note that Assumption~\ref{ass:D^Y_full_rank} is imposed primarily to simplify the notation and presentation of the analysis.  By applying the conversion framework developed in \cite{chen2026relationships}, all results stated under Assumption~\ref{ass:D^Y_full_rank} extend to the case where the columns of $D^Y$ span a linear subspace. In that setting, the analysis can be carried out in intrinsic coordinates on the subspace, after which the resulting bounds for the subspace model can be translated back to the full-space model on $Y$ via the conversion formula in \cite{chen2026relationships}.

    Assumption \ref{ass:C1+} is standard in error analysis for model-based DFO, as it provides a bound on the error of the first-order Taylor approximation.
    \begin{lemma}{(\cite[Lem.~9.4]{audet2017derivative})}
        Suppose that Assumption \ref{ass:C1+} holds.  Then, for all $x,d\in\mathbb{R}^n$, we have $|f(x+d)-f(x)-\nabla f(x)^\top d|\le\frac{L_{\nabla f}}{2}\|d\|^2$.
    \end{lemma}

    The next theorem, analogous to \cite[Thm.~5.4]{conn2009introduction}, provides a fully linear bound for general quadratic interpolation models. 
    \begin{theorem}\label{thm:QuadfullylinearConsts}
        Suppose that Assumptions~\ref{ass:D^Y_full_rank} and \ref{ass:C1+} hold. Let $m^{x^0,Y}(x)$ be any quadratic interpolation model of $f$ at $x^0$ over $Y$.  If $\|\nabla^2m^{x^0,Y}(x)\|\le \kappa^{mH}$ for some $\kappa^{mH}>0$ on a set $K\subseteq\mathbb{R}^n$, then for all $x^0\in K$, $m^{x^0,Y}(x)$ belongs to a class of fully linear models of $f$ at $x^0$ parameterized by $\Delta_{D^Y}$, with
        \begin{equation*}
            \kappa^{ef} := \frac{L_{\nabla f}+\kappa^{mH}}{2} \sqrt{n}\left\|\overline{D^Y}^{\dagger}\right\|_1 + \frac{L_{\nabla f}+\kappa^{mH}}{2},~~~\kappa^{eg} := 2\kappa^{ef} + 2\kappa^{mH}.
        \end{equation*}
    \end{theorem}
    \begin{proof}
        Fix $x\in B(x^0;\Delta_{D^Y})$ and denote $d=x-x^0$.
        From Assumption~\ref{ass:D^Y_full_rank}, there exist $v\in\mathbb{R}^m$ such that $d = \sum_{i=1}^m v_i d^i$.
        This implies that $\Delta_{D^Y}\overline{D^Y} v = d$ is a consistent linear system.  Let $v$ be the minimum norm solution to this system, i.e., $v = \overline{D^Y}^{\dagger} d\slash\Delta_{D^Y}$.  Suppose that $m^{x^0,Y}(x)=f(x^0)+\alpha^\top d+\frac{1}{2}d^\top Hd$.  Since $m^{x^0,Y}(x)$ interpolates $f$ on $Y$, $\alpha^\top d^i=f(x^0+d^i)-f(x^0)-\frac{1}{2}(d^i)^\top Hd^i$ and so
        \begin{align*}
            &~~\left|m^{x^0,Y}(x)-f(x^0)-\nabla f(x^0)^\top d\right|\\
            &\le \left|\left(\alpha-\nabla f(x^0)\right)^\top d\right| + \frac{1}{2}\left|d^\top Hd\right| \le \left|\sum_{i=1}^mv_i\left(\alpha-\nabla f(x^0)\right)^\top d^i\right| + \frac{\kappa^{mH}}{2}\Delta_{D^Y}^2\\
            &\le \left|\sum_{i=1}^mv_i\left(f(x^0+d^i)-f(x^0)-\nabla f(x^0)^\top d^i\right)\right| + \frac{\kappa^{mH}}{2}\left(\left\|v\right\|_1+1\right)\Delta_{D^Y}^2\\
            &\le \left(\frac{L_{\nabla f}+\kappa^{mH}}{2}\left\|v\right\|_1+\frac{\kappa^{mH}}{2}\right)\Delta_{D^Y}^2.
        \end{align*}
        Since $\|d\|_1 \leq \sqrt{n} \|d\|$, we have $\|v\|_1 \le \|\overline{D^Y}^{\dagger}\|_1\|d\|_1\slash\Delta_{D^Y} \le \sqrt{n} \|\overline{D^Y}^{\dagger}\|_1$. The result follows from \cite[Lem.~5.1(a)]{roberts2025introduction}.
    \end{proof}
    \begin{remark}
        Theorem~\ref{thm:QuadfullylinearConsts} implies that any quadratic interpolation model with a bounded Hessian belongs to a class of fully linear models.  Therefore, to prove that a class of models is fully linear, it suffices to prove, under reasonable assumptions, that the corresponding Hessians are bounded. 
    \end{remark}

\subsection{Error bounds of MFN models}
    For the MFN model, fully linear error bounds can be found in, e.g., \cite[Thm.~5.15]{roberts2025introduction} and \cite[Thms.~5.4 and 5.7]{conn2009introduction}.  However, their definitions of MFN models differ slightly from Definition~\ref{def:MFN}. In particular, Definition~\ref{def:MFN} requires the model to interpolate $f$ at $x^0$, which is not imposed in \cite[Thm.~5.15]{roberts2025introduction}.  Moreover, \cite[Thm.~5.15]{roberts2025introduction} replaces the objective in Problem~\eqref{pro:MFN} with $\frac{1}{4}\|H\|_F$, and \cite[Thms.~5.4 and 5.7]{conn2009introduction} express the bounds in terms of a specific geometric measure of $Y$ known as the $\Lambda$-poisedness constant. Following the approach of \cite[Thm.~5.15]{roberts2025introduction}, we establish a fully linear error bound for the MFN model in Definition~\ref{def:MFN}.

    \begin{theorem}\label{thm:MFNHessupperbd}
        Suppose that Assumption \ref{ass:C1+} holds and $Y$ is poised for MFN interpolation.  Define $P\in\mathbb{R}^{m\times m}$ with $P_{ij}:=\frac{1}{4}((d^i\slash\Delta_{D^Y})^\top(d^j\slash\Delta_{D^Y}))^2$, and $F := \begin{bmatrix}
                P & \overline{D^Y}^\top \\
                \overline{D^Y} & \mymathbb{0}_{n\times n}
            \end{bmatrix}$.
        Then, $\nabla_{\mathrm{MFN}}^2f(x^0;Y)$ is unique and satisfies
        \begin{equation*}
            \left\|\nabla_{\mathrm{MFN}}^2f(x^0;Y)\right\| \le \left\|\nabla_{\mathrm{MFN}}^2f(x^0;Y)\right\|_F \le \kappa^{mH}_{\mathrm{MFN}}:= \frac{L_{\nabla f}}{4}m\left\|F^{-1}\right\|_{\infty}.
        \end{equation*}
    \end{theorem}
    \begin{proof}
        Similar to \cite[Lem.~5.11]{roberts2025introduction}, we can show that the solution to Problem~\eqref{pro:MFN}, denoted by $(\alpha^*,H^*)$, can be obtained by solving
        \begin{equation*}
            \underbrace{\begin{bmatrix}
                \Delta_{D^Y}^4 P & \left(D^Y\right)^\top\\
                D^Y & \mymathbb{0}_{n\times n}
            \end{bmatrix}}_{\widetilde{F}}\begin{bmatrix}
                \lambda\\
                \alpha^*
            \end{bmatrix} = \begin{bmatrix}
                \delta_f(x^0;D^Y)\\
                \mymathbb{0}
            \end{bmatrix}~~~\text{and}~~~H^*=\frac{1}{2}\sum_{i=1}^m\lambda_id^i(d^i)^\top,
        \end{equation*}
        where $\lambda\in\mathbb{R}^m$ is the Lagrange multiplier.  Since $Y$ is poised for MFN interpolation, $\widetilde{F}$ is invertible and so $F$ is invertible.  The proof of \cite[Lem.~5.14]{roberts2025introduction} gives $|\lambda_i| \le \frac{L_{\nabla f}}{2}\|F^{-1}\|_{\infty}\Delta_{D^Y}^{-2}$, so
        \begin{align*}
            \left\|H^*\right\| \le \left\|H^*\right\|_F \le \frac{1}{2}\sum_{i=1}^m\left|\lambda_i\right|\left\|d^i\right\|^2 \le \frac{1}{2}\sum_{i=1}^m\left|\lambda_i\right|\Delta_{D^Y}^2 \le \frac{L_{\nabla f}}{4}m\left\|F^{-1}\right\|_{\infty}.
        \end{align*}
    \end{proof}
    
    \begin{theorem}\label{thm:MFNfullylinear}
        Under the assumptions of Theorems ~\ref{thm:QuadfullylinearConsts} and \ref{thm:MFNHessupperbd}, $m_{\mathrm{MFN}}^{x^0,Y}(x)$ belongs to a class of fully linear models of $f$ at $x^0$ parameterized by $\Delta_{D^Y}$, with
        \begin{equation*}
            \kappa^{ef}_{\mathrm{MFN}} := \frac{L_{\nabla f}+\kappa^{mH}_{\mathrm{MFN}}}{2}\sqrt{n}\left\|\overline{D^Y}^{\dagger}\right\|_1 + \frac{L_{\nabla f}+\kappa^{mH}_{\mathrm{MFN}}}{2},~~~\kappa^{eg}_{\mathrm{MFN}} := 2\kappa^{ef}_{\mathrm{MFN}} + 2\kappa^{mH}_{\mathrm{MFN}}.
        \end{equation*}
    \end{theorem}
    \begin{proof}
        Directly apply Theorem~\ref{thm:QuadfullylinearConsts} with $\kappa^{mH}=\kappa^{mH}_{\mathrm{MFN}}$.
    \end{proof}

\subsection{Error bounds of MN models}
    For the MN model, a fully linear bound follows from \cite[Thm.~5.4]{conn2009introduction} if the model Hessians are uniformly bounded. In this section, we show that this assumption is unnecessary, provided that the sample set is poised for MFN interpolation.
    \begin{remark}
        In practice, we often assume that $\{x\in\mathbb{R}^n:f(x)\le f(x^0)\}$ is compact.  This is a common choice for the set $K$ in the following theorems.
    \end{remark}
    \begin{theorem}\label{thm:MNHessupperbd}
        Suppose that Assumption \ref{ass:C1+} holds and $Y$ is poised for MFN interpolation.  Let $K\subseteq\mathbb{R}^n$ be a compact set and so there exists $\kappa^g>0$ such that
        $\|\nabla f(x)\|\le\kappa^g$ on $K$.  Then, $\nabla_{\mathrm{MN}}^2f(x^0;Y)$ is unique and satisfies
        \begin{equation*}
            \left\|\nabla_{\mathrm{MN}}^2f(x^0;Y)\right\| \le \kappa^{mH}_{\mathrm{MN}}:=\sqrt{\left(\kappa^g+\kappa^{eg}_{\mathrm{MFN}}\overline{\Delta}\right)^2 + \left(\kappa^{mH}_{\mathrm{MFN}}\right)^2}~~~\text{for all $x^0\in K$},
        \end{equation*}
        for any $\overline{\Delta}\geq \Delta_{D^Y}$. 
    \end{theorem}
    \begin{proof}
        Since  $(\nabla_{\mathrm{MFN}}f(x^0;Y),\nabla_{\mathrm{MFN}}^2f(x^0;Y))$ is feasible for Problem~\eqref{pro:MN},
        \begin{align*}
            \left\|\nabla_{\mathrm{MN}}^2f(x^0;Y)\right\|^2 &\le \left\|\nabla_{\mathrm{MN}}^2f(x^0;Y)\right\|_F^2 \le \left\|\nabla_{\mathrm{MN}}f(x^0;Y)\right\|^2+\left\|\nabla_{\mathrm{MN}}^2f(x^0;Y)\right\|_F^2\\ &\le \left\|\nabla_{\mathrm{MFN}}f(x^0;Y)\right\|^2+\left\|\nabla_{\mathrm{MFN}}^2f(x^0;Y)\right\|_F^2.
        \end{align*}
        Using the triangle inequality and the fact that $m_{\mathrm{MFN}}^{x^0,Y}(x)$ is fully linear, we get $\|\nabla_{\mathrm{MFN}}f(x^0;Y)\| \le \|\nabla f(x^0)\| + \kappa^{eg}_{\mathrm{MFN}}\Delta_{D^Y}$.  The proof is complete by combining the inequalities above and noticing that $\|\nabla f(x)\|\le\kappa^g$ on $K$.
    \end{proof}
    
    \begin{theorem}\label{thm:MNfullylinear}
        Under the assumptions of Theorems \ref{thm:QuadfullylinearConsts} and \ref{thm:MNHessupperbd}, for all $x^0\in K$, $m_{\mathrm{MN}}^{x^0,Y}(x)$ belongs to a class of fully linear models of $f$ at $x^0$ parameterized by $\Delta_{D^Y}$, with
        \begin{equation*}
            \kappa^{ef}_{\mathrm{MN}} := \frac{L_{\nabla f}+\kappa^{mH}_{\mathrm{MN}}}{2} \sqrt{n} \left\|\overline{D^Y}^{\dagger}\right\|_1 + \frac{L_{\nabla f}+\kappa^{mH}_{\mathrm{MN}}}{2},~~~\kappa^{eg}_{\mathrm{MN}} := 2\kappa^{ef}_{\mathrm{MN}} + 2\kappa^{mH}_{\mathrm{MN}}.
        \end{equation*}
    \end{theorem}
    \begin{proof}
        Directly apply Theorem~\ref{thm:QuadfullylinearConsts} with $\kappa^{mH}=\kappa^{mH}_{\mathrm{MN}}$.
    \end{proof}

\subsection{Error bounds of QS models}
    In this subsection, we derive a fully linear error bound for the QS model in the case where it is an interpolation model.   To our knowledge, this is the first fully linear bound for the QS model. 

    \begin{theorem}\label{thm:QSHessupperbd}
        Suppose that Assumption~\ref{ass:C1+} holds and $\nabla^2_{\mathrm{QS}}f(x^0;Y)$ has the form $\nabla^2_{\mathrm{QS}}f(x^0;Y)=\sum_{k=1}^l\beta_k\nabla_s^2f(x^0;S^k;T^k_{1:p^k})$, where $l\ge 1$, $S^k\in\mathbb{R}^{n\times p^k}$, and $T^k_{1:p^k}=\{T_1^k,\ldots,T_{p^k}^k\}$ with each $T_i^k\in\mathbb{R}^{n\times q_i^k}$. Then,
        \begin{equation*}
            \left\|\nabla^2_{\mathrm{QS}}f(x^0;Y)\right\| \le \kappa^{mH}_{\mathrm{QS}}:=L_{\nabla f}\sum_{k=1}^l\left|\beta_k\right|\left\|\overline{S^k}^\dagger\right\|\sqrt{\sum_{i=1}^{p^k} q_i^k\left\|\overline{T_i^k}^\dagger\right\|^2}.
        \end{equation*}
    \end{theorem}
    \begin{proof}
        It suffices to show $\|\nabla^2_sf(x^0;S;T_{1:p})\| \le L_{\nabla f}\|\overline{S}^\dagger\|\sqrt{\sum_{i=1}^p q_i\|\overline{T_i}^\dagger\|^2}$ for any $S\in\mathbb{R}^{n\times p}$ and $T_i\in\mathbb{R}^{n\times q_i}$, $i\in[p]$, since the result then follows from the triangle inequality.  The $j$-th element of $\delta_f(x^0+s^i;T_i)-\delta_f(x^0;T_i)$ satisfies
        \begin{align*}
            &~~\left|f(x^0+s^i+t_i^j)-f(x^0+s^i)-f(x^0+t_i^j)+f(x^0)\right|\\
            &= \left|\int_0^1\left(\nabla f(x^0+s^i+\tau t_i^j)-\nabla f(x^0+\tau t_i^j)\right)^\top t_i^j~\mathrm{d}\tau\right| \le L_{\nabla f}\Delta_S\Delta_{T_i}.
        \end{align*}
        We have the $i$-th row of $\delta_f^2(x^0;S;T_{1:p})$ satisfies
        \begin{align*}
            &~~\left\|\nabla_sf(x^0+s^i;T_i)-\nabla_sf(x^0;T_i)\right\| = \left\|\left(T_i^\top\right)^\dagger\left(\delta_f(x^0+s^i;T_i)-\delta_f(x^0;T_i)\right)\right\|\\
            &\le \sqrt{q_i}L_{\nabla f}\Delta_S\Delta_{T_i}\left\|T_i^\dagger\right\| = \sqrt{q_i}L_{\nabla f}\left\|\overline{T_i}^\dagger\right\|\Delta_S.
        \end{align*}
        The result follows from
        \begin{equation*}
            \left\|\nabla^2_sf(x^0;S;T_{1:p})\right\| \le \left\|S^\dagger\right\|\left\|\delta_f^2(x^0;S;T_{1:p})\right\|_F \le \left\|S^\dagger\right\|L_{\nabla f}\Delta_S\sqrt{\sum_{i=1}^p q_i\left\|\overline{T_i}^\dagger\right\|^2}.
        \end{equation*}
    \end{proof}
    \begin{theorem}\label{thm:QSfullylinear}
        Suppose that the assumptions of Theorems~\ref{thm:QuadfullylinearConsts} and \ref{thm:QSHessupperbd} hold.  If $m_{\mathrm{QS}}^{x^0,Y}(x)$ interpolates $f$ on $Y$, then for all $x^0\in\mathbb{R}^n$, $m_{\mathrm{QS}}^{x^0,Y}(x)$ belongs to a class of fully linear models of $f$ at $x^0$ parameterized by $\Delta_{D^Y}$, with
        \begin{equation*}
            \kappa^{ef}_{\mathrm{QS}} := \frac{L_{\nabla f}+\kappa^{mH}_{\mathrm{QS}}}{2} \sqrt{n} \left\|\overline{D^Y}^{\dagger}\right\|_1 + \frac{L_{\nabla f}+\kappa^{mH}_{\mathrm{QS}}}{2},~~~\kappa^{eg}_{\mathrm{QS}} := 2\kappa^{ef}_{\mathrm{QS}} + 2\kappa^{mH}_{\mathrm{QS}}.
        \end{equation*}
    \end{theorem}
    \begin{proof}
        Directly apply Theorem~\ref{thm:QuadfullylinearConsts} with $\kappa^{mH}=\kappa^{mH}_{\mathrm{QS}}$.
    \end{proof}

\section{Directional Hessian error bounds with structured $Y$}\label{sec:directionalHessebs}
    In this section, we derive directional Hessian error bounds for the Hessian approximations arising from the MN, MFN, and QS models, under a mild structural assumption on $Y$.  Our results show that the directional errors range from $\mathcal{O}(1)$ to $\mathcal{O}(\Delta_{D^Y})$, depending on the direction of interest.  Notably, they achieve fully quadratic accuracy along sample directions. This highlights that all three models achieve significantly higher accuracy than general fully linear models in a directional sense.  To our knowledge, these are the first directional error bounds established for these Hessian approximations.

    We may use the following assumptions.
    \begin{assumption}\label{ass:structuredY}
        The set $Y=\{x^0,x^0+d^i,x^0-d^i:i\in [p]\}$.
    \end{assumption}
    \begin{assumption}\label{ass:fC2+}
        The objective $f\in\mathcal{C}^{2+}$ with $L_{\nabla^2 f}$-Lipschitz Hessian. 
    \end{assumption}
    \begin{assumption}\label{ass:MNfeasible}
        The Problem~\eqref{pro:MN} is feasible.
    \end{assumption}

    Notice that Assumption~\ref{ass:MNfeasible} guarantees the existence of a solution to Problem~\eqref{pro:MFN}.  However, the solution need not be unique.  Indeed, as shown in \cite[Rem.~2 and Exp.~1]{chen2026relationships}, in this case, $\nabla^2_{\mathrm{MFN}}f(x^0;Y)$ is always unique, whereas $\nabla_{\mathrm{MFN}}f(x^0;Y)$ may not be unique.

    The following two lemmas present basic properties of $\mathcal{C}^{2+}$ functions and will be used multiple times in the proofs in this subsection.
    \begin{lemma}
        Suppose that Assumption~\ref{ass:fC2+} holds.  For $x^0,u,v,w^1,w^2\in\mathbb{R}^n$, define $F(u,v,w^1,w^2) := \int_{0}^{1} (1-t) u^\top \left(\nabla^2 f(x^0+tw^1)-\nabla^2 f(x^0+tw^2)\right) v~\mathrm{d}t$.  Then, $|F(u,v,w^1,w^2)| \le \frac{L_{\nabla^2 f}}{6}\|u\|\|v\|\|w^1-w^2\|$.
    \end{lemma}
    \begin{proof}
        By the triangle inequality and Assumption~\ref{ass:fC2+}, we obtain
        \begin{equation*}
            \left|F(u,v,w^1,w^2)\right|\le L_{\nabla^2 f}\|u\|\|v\|\left\|w^1-w^2\right\|\int_{0}^{1} (1-t)t~\mathrm{d}t,
        \end{equation*}
        and the result follows.
    \end{proof}
    \begin{lemma}\label{lem:C2+GivesO(D3)}
        Suppose that Assumption~\ref{ass:fC2+} holds.  Let $\Delta>0$.  Then, for any $x^0\in\mathbb{R}^n$, $u,v\in B(\mymathbb{0};\Delta)\subseteq\mathbb{R}^n$,
        \begin{align*}
            &~~E(u,v):=\\
            &\left|2f(x^0+\frac{1}{2}u+\frac{1}{2}v)+2f(x^0-\frac{1}{2}u-\frac{1}{2}v)-\frac{1}{2}f(x^0+u)-\frac{1}{2}f(x^0+v)\right.\\
            &\hspace{1.6cm}\left.-\frac{1}{2}f(x^0-u)-\frac{1}{2}f(x^0-v)-2f(x^0)-u^\top\nabla^2f(x^0)v\right| \le \frac{2L_{\nabla^2 f}}{3}\Delta^3.
        \end{align*}
        Moreover,
        \begin{equation*}
            E(u,u) = \left|f(x^0+u)+f(x^0-u)-2f(x^0)-u^\top\nabla^2f(x^0)u\right| \le \frac{L_{\nabla^2 f}}{3}\Delta^3.
        \end{equation*}
    \end{lemma}
    \begin{proof}
        Notice that $u^\top\nabla^2 f(x^0)v = 2\int_{0}^{1} (1-t) u^\top \nabla^2 f(x^0) v~\mathrm{d}t$ and, for all $d\in\mathbb{R}^n$, $f(x^0 + d) = f(x^0) + \nabla f(x^0)^\top d + \int_{0}^{1} (1-t) d^\top \nabla^2 f(x^0+td) d~\mathrm{d}t$.
        We can compute that 
        \begin{align*}
            &~~E(u,v)\\
            &= \left|\frac{1}{2}F(u,u,\frac{1}{2}(u+v),u) + \frac{1}{2}F(v,v,\frac{1}{2}(u+v),v) + \frac{1}{2}F(u,u,-\frac{1}{2}(u+v),-u)\right.\\
            &\hspace{0.6cm}\left.+ \frac{1}{2}F(v,v,-\frac{1}{2}(u+v),-v) + F(u,v,\frac{1}{2}(u+v),\mymathbb{0}) + F(u,v,-\frac{1}{2}(u+v),\mymathbb{0})\right|\\
            &\le \frac{L_{\nabla^2 f}}{6}\left(\frac{1}{2}\|u\|^2\|u-v\| + \frac{1}{2}\|v\|^2\|u-v\| + \|u\|\|v\|\|u+v\|\right) \le \frac{2L_{\nabla^2 f}}{3}\Delta^3.
        \end{align*}
        Finally, if we take $u=v$ in the above inequality, then $\|u-v\|=0$ and so
        $E(u,u) \le \frac{L_{\nabla^2 f}}{6}\|u\|\|v\|\|u+v\| \le \frac{L_{\nabla^2 f}}{3}\Delta^3$.
    \end{proof}

    We begin with a directional Hessian error bound in the directions $d^i\slash\|d^i\|$.  Lemma~\ref{lem:|di(MNMFNHess-Hess)di|} shows that this bound is $\mathcal{O}(\Delta_{D^Y})$ for the MN and MFN models.
    \begin{lemma}\label{lem:|di(MNMFNHess-Hess)di|}
        Suppose that Assumptions~\ref{ass:structuredY} to \ref{ass:MNfeasible} hold.  Let $\mathcal{X}\in\{\mathrm{MN},\mathrm{MFN}\}$. Then, for any non-zero $d^i\in D^Y$,
        \begin{equation}\label{ineq:|di(MNMFNHess-Hess)di|}
            \left|\left(\frac{d^i}{\|d^i\|}\right)^\top\left(\nabla_{\mathcal{X}}^2f(x^0;Y)-\nabla^2f(x^0)\right)\frac{d^i}{\|d^i\|}\right| \le \frac{L_{\nabla^2f}}{3}\Delta_{D^Y}.
        \end{equation}
    \end{lemma}
    \begin{proof}
        Under Assumption~\ref{ass:structuredY}, direct row operations show that the equality constraints of Problems~\eqref{pro:MN} and \eqref{pro:MFN} are 
        \begin{align*}
        \begin{cases}
            &\left(d^i\right)^\top\alpha = \frac{1}{2}\left(f(x^0+d^i)-       f(x^0)\right)-\frac{1}{2}\left(f(x^0-d^i)-       f(x^0)\right),~~~i\in [p],\\
            &\left(d^i\right)^\top H\left(d^i\right) = f(x^0+d^i) + f(x^0-d^i)-2f(x^0),~~~i\in [p].
        \end{cases}
        \end{align*} 
        Using Lemma~\ref{lem:C2+GivesO(D3)} with $u=d^i$ and $\Delta=\|d^i\|$, we obtain
        \begin{align*}
            \left|\left(\frac{d^i}{\|d^i\|}\right)^\top\left(\nabla_{\mathcal{X}}^2f(x^0;Y)-\nabla^2f(x^0)\right)\frac{d^i}{\|d^i\|}\right| = \frac{1}{\|d^i\|^2}E(d^i,d^i) \le \frac{L_{\nabla^2f}}{3}\Delta_{D^Y}.
        \end{align*}
    \end{proof}

    If we replace one of the $d^i\slash\|d^i\|$ on the left-hand side of Inequality \eqref{ineq:|di(MNMFNHess-Hess)di|} by $d^j\slash\|d^j\|$ where $i\neq j$, then the error bound is slightly worse than that in Inequality \eqref{ineq:|di(MNMFNHess-Hess)di|}.  As shown in Lemma~\ref{lem:|di(MNMFNHess-Hess)dj|}, the bound in this case is $\mathcal{O}(1)$.
    \begin{lemma}\label{lem:|di(MNMFNHess-Hess)dj|}
        Suppose that Assumptions \ref{ass:D^Y_full_rank} to \ref{ass:MNfeasible} hold and $Y$ is poised for MFN interpolation.
        Let $\mathcal{X}\in\{\mathrm{MN},\mathrm{MFN}\}$. Then, for any non-zero $d^i,d^j\in D^Y$ with $i\neq j$,
        \begin{equation*}
            \left|\left(\frac{d^i}{\|d^i\|}\right)^\top\left(\nabla_{\mathcal{X}}^2f(x^0;Y)-\nabla^2f(x^0)\right)\frac{d^j}{\|d^j\|}\right| \le 4\kappa^{ef}_{\mathcal{X}}\frac{\Delta_{D^Y}^2}{\|d^i\|\|d^j\|} + \frac{2L_{\nabla^2 f}}{3}\frac{\Delta_{D^Y}^3}{\|d^i\|\|d^j\|}.
        \end{equation*}
    \end{lemma}
    \begin{proof}
        Under the assumptions, we have that $m^{x^0,Y}_{\mathcal{X}}(x)$ is fully linear by Theorems \ref{thm:MFNfullylinear} and \ref{thm:MNfullylinear}.  Since $(\nabla_{\mathcal{X}}f(x^0;Y),\nabla^2_{\mathcal{X}}f(x^0;Y))$ satisfies the constraints of Problems~\eqref{pro:MN} and~\eqref{pro:MFN}, we can compute that
        \begin{align*}
            &~~\left(d^i\right)^\top\nabla^2_{\mathcal{X}}f(x^0;Y)d^j\\
            &= 2m^{x^0,Y}_{\mathcal{X}}(x^0+\frac{1}{2}d^i+\frac{1}{2}d^j)+2m^{x^0,Y}_{\mathcal{X}}(x^0-\frac{1}{2}d^i-\frac{1}{2}d^j)\\
            &\hspace{1cm}-\frac{1}{2}f(x^0+d^i)-\frac{1}{2}f(x^0+d^j)-\frac{1}{2}f(x^0-d^i)-\frac{1}{2}f(x^0-d^j)-2f(x^0).
        \end{align*}
        Using the triangle inequality and the fact that $m^{x^0,Y}_{\mathcal{X}}(x)$ is fully linear, we get
        \begin{align*}
            \left|\left(d^i\right)^\top\left(\nabla_{\mathcal{X}}^2f(x^0;Y)-\nabla^2f(x^0)\right)d^j\right| \le 4\kappa^{ef}_{\mathcal{X}}\Delta_{D^Y}^2 + E(d^i,d^j).
        \end{align*}
        The result follows from Lemma~\ref{lem:C2+GivesO(D3)} with $u=d^i$, $v=d^j$, and $\Delta=\Delta_{D^Y}$.
    \end{proof}

    Using Lemmas~\ref{lem:|di(MNMFNHess-Hess)di|} and \ref{lem:|di(MNMFNHess-Hess)dj|}, we establish the following bounds on the directional Hessian error for any direction $d\slash\|d\|\in\mathbb{R}^n$.  These results further yield a $\mathcal{O}(1)$ bound on the Hessian approximation error.
    \begin{theorem}\label{thm:|d(MNMFNHess-Hess)d|}
        Suppose that the assumptions of Lemma~\ref{lem:|di(MNMFNHess-Hess)dj|} hold.  Denote $D:=[d^1\cdots d^p]$.  Then, $\mathrm{col}(D)=\mathbb{R}^n$ and for any non-zero $d\in\mathbb{R}^n$, we can write $d=Dv$ with $v=D^\dagger d$.  Moreover,
        \begin{align*}
            &~~\left|\left(\frac{d}{\|d\|}\right)^\top\left(\nabla_{\mathcal{X}}^2f(x^0;Y)-\nabla^2f(x^0)\right)\frac{d}{\|d\|}\right|\\
            &\le 4\frac{\|v\|_1^2-\|v\|_{\infty}^2}{\|v\|^2}\kappa^{ef}_{\mathcal{X}}\left\|\overline{D}^\dagger\right\|^2+\frac{L_{\nabla^2 f}}{3}\left\|\overline{D}^\dagger\right\|^2\left(2\frac{\|v\|_1^2-\|v\|_{\infty}^2}{\|v\|^2}+1\right)\Delta_D.
        \end{align*}
        In particular, we have that
        \begin{align*}
            &~~\left\|\nabla_{\mathcal{X}}^2f(x^0;Y)-\nabla^2f(x^0)\right\| = \max_{\|d\|=1}\left|d^\top\left(\nabla_{\mathcal{X}}^2f(x^0;Y)-\nabla^2f(x^0)\right)d\right|\\
            &\le 4\left(p-\frac{1}{p}\right)\kappa^{ef}_{\mathcal{X}}\left\|\overline{D}^\dagger\right\|^2 + \frac{L_{\nabla^2 f}}{3}\left\|\overline{D}^\dagger\right\|^2\left(2p-\frac{2}{p}+1\right)\Delta_D.
        \end{align*}
        
    \end{theorem}
    \begin{proof}
        Under Assumption \ref{ass:D^Y_full_rank}, we  have that $\mathrm{col}(D)=\mathrm{col}(D^Y)=\mathbb{R}^n$ and $d=DD^\dagger d=Dv=\sum_{i=1}^pv_id^i$.  Moreover, $\Delta_D=\Delta_{D^Y}$.  Combining the triangle inequality and previous results in this section, we get
        \begin{align*}
            &~~\left|\left(\frac{d}{\|d\|}\right)^\top\left(\nabla_{\mathcal{X}}^2f(x^0;Y)-\nabla^2f(x^0)\right)\frac{d}{\|d\|}\right|\\
            &= \frac{1}{\|d\|^2}\left|\left(\sum_{i=1}^pv_i\left(d^i\right)^\top\right)\left(\nabla_{\mathcal{X}}^2f(x^0;Y)-\nabla^2f(x^0)\right)\left(\sum_{i=1}^pv_id^i\right)\right|\\
            &\le \frac{L_{\nabla^2 f}}{3}\frac{\|v\|^2}{\|d\|^2}\Delta_D^3 + \frac{1}{\|d\|^2}\left(4\kappa^{ef}_{\mathcal{X}}\Delta_D^2 + \frac{2L_{\nabla^2 f}}{3}\Delta_D^3\right)\sum_{i,j=1, i\neq j}^p\left|v_iv_j\right|.
        \end{align*}
        Without loss of generality, we suppose that $v_k$ has the largest absolute value among all $v_i$, where $i\in [p]$.  Then,
        \begin{align*}
            \sum_{i,j=1, i\neq j}^p |v_i v_j| &= 2\left|v_k\right|\sum_{i=1, i\neq k}^p\left|v_i\right|+\left(\sum_{i=1, i\neq k}^p\left|v_i\right|\right)^2-\sum_{i=1, i\neq k}^p\left|v_i\right|^2\\
            &\le 2\left|v_k\right|\sum_{i=1, i\neq k}^p\left|v_i\right| + \left(\sum_{i=1, i\neq k}^p\left|v_i\right|\right)^2\\
            &= 2\left\|v\right\|_{\infty}\left(\left\|v\right\|_1-\left\|v\right\|_{\infty}\right)+\left(\left\|v\right\|_1-\left\|v\right\|_{\infty}\right)^2 = \left\|v\right\|_1^2-\left\|v\right\|_{\infty}^2.
        \end{align*}
        Notice that $\|\overline{D}^\dagger\|=\|D^\dagger\|\Delta_D$, and $\|v\|\le\|D^\dagger\|\|d\|$ implies $\|d\|\ge\|v\|\slash\|D^\dagger\|$.  We obtain
        \begin{align*}
            &~~\left|\left(\frac{d}{\|d\|}\right)^\top\left(\nabla_{\mathcal{X}}^2f(x^0;Y)-\nabla^2f(x^0)\right)\frac{d}{\|d\|}\right|\\
            &\le \frac{L_{\nabla^2 f}}{3}\frac{\|v\|^2}{\|d\|^2}\Delta_D^3 + \left(4\kappa^{ef}_{\mathcal{X}}\Delta_D^2 + \frac{2L_{\nabla^2 f}}{3}\Delta_D^3\right)\frac{\|v\|_1^2-\|v\|_{\infty}^2}{\|d\|^2}\\
            &\le 4\frac{\|v\|_1^2-\|v\|_{\infty}^2}{\|v\|^2}\kappa^{ef}_{\mathcal{X}}\left\|\overline{D}^\dagger\right\|^2+\frac{L_{\nabla^2 f}}{3}\left\|\overline{D}^\dagger\right\|^2\left(2\frac{\|v\|_1^2-\|v\|_{\infty}^2}{\|v\|^2}+1\right)\Delta_D.
        \end{align*}

        Finally, we take the maximum over all $d$ with $\|d\|=1$, which is equivalent to maximizing over all $v\in\mathrm{col}(D^\top)$ with $\|Dv\|=1$.  Since we are only seeking an upper bound, we ignore the constraint $v\in\mathrm{col}(D^\top)$. Let $w=v\slash\|v\|$.  Then,
        \begin{align*}
            \max_{\|Dv\|=1}\frac{\|v\|_1^2-\|v\|_{\infty}^2}{\|v\|^2} = \max_{\|D w\|\|v\|=1}\left(\|w\|_1^2-\|w\|_{\infty}^2\right) = \max_{\substack{\|w\|=1\\ \|Dw\|>0}}\left(\|w\|_1^2-\|w\|_{\infty}^2\right).
        \end{align*}
        Notice that the set $\argmax_{\|w\|=1}(\|w\|_1^2-\|w\|_{\infty}^2) = \{\frac{1}{\sqrt{p}}[\pm 1,\ldots,\pm 1]^\top\}\subseteq\mathbb{R}^p$ spans $\mathbb{R}^p$.  Since $\mathrm{rank}(D)=n>0$, there must exist an element in the set such that $\|Dw\|>0$.  This gives
        \begin{equation*}
            \max_{\|w\|=1,\|Dw\|>0}\left(\|w\|_1^2-\|w\|_{\infty}^2\right) = \max_{\|w\|=1}\left(\|w\|_1^2-\|w\|_{\infty}^2\right) = p-\frac{1}{p}
        \end{equation*}
        and so
        \begin{align*}
            &~~\left\|\nabla_{\mathcal{X}}^2f(x^0;Y)-\nabla^2f(x^0)\right\| =\max_{\|d\|=1}\left|d^\top\left(\nabla_{\mathcal{X}}^2f(x^0;Y)-\nabla^2f(x^0)\right)d\right|\\
            &\le 4\left(p-\frac{1}{p}\right)\kappa^{ef}_{\mathcal{X}}\left\|\overline{D}^\dagger\right\|^2 + \frac{L_{\nabla^2 f}}{3}\left\|\overline{D}^\dagger\right\|^2\left(2p-\frac{2}{p}+1\right)\Delta_D.
        \end{align*}
    \end{proof}
    
    \begin{remark}
        Comparing Lemmas \ref{lem:|di(MNMFNHess-Hess)di|}, \ref{lem:|di(MNMFNHess-Hess)dj|}, and Theorem \ref{thm:|d(MNMFNHess-Hess)d|}, we see that the Hessian approximations arising from the MN and MFN models have directional errors ranging from $\mathcal{O}(1)$ to $\mathcal{O}(\Delta_{D^Y})$.  In particular, along the directions $d^i\slash\|d^i\|$, the error is $\mathcal{O}(\Delta_{D^Y})$, indicating that the MN and MFN models attain the same level of accuracy as fully quadratic models in these directions.
    \end{remark}
    \begin{remark}
        In Lemma~\ref{lem:|di(MNMFNHess-Hess)dj|} and Theorem~\ref{thm:|d(MNMFNHess-Hess)d|}, Assumptions~\ref{ass:D^Y_full_rank} and \ref{ass:C1+}, together with the requirement that $Y$ be poised for MFN interpolation, are imposed solely to ensure that $m^{x^0,Y}_{\mathcal{X}}(x)$ is fully linear.  However, an inspection of the proofs shows that the conclusions remain valid if these assumptions are replaced by the weaker requirement that $m^{x^0,Y}_{\mathcal{X}}(x)$ be fully linear only on the affine subspace $x^0 + \mathrm{col}(D^Y)$ where $\mathrm{col}(D^Y)$ is not necessarily $\mathbb{R}^n$.
        
        A detailed study of fully linear models defined on affine subspaces can be found in \cite{cartis2023scalable,chen2024qfully}, while a thorough analysis of the relationship between full-space and subspace models is provided in \cite{chen2026relationships}.
    \end{remark}

    Now, we develop similar results for the GSH.  The results can be applied to derive similar directional error bounds for the Hessian of QS models.  We will use the following, more specific version of Assumption~\ref{ass:structuredY}.
    \begin{assumption}\label{ass:structuredS&Ti}
        The matrix $S=[d^1\cdots d^p]$ has full column rank and for $i\in [p]$, $T_i = [-d^i]$.  
    \end{assumption}

    We have $\Delta_{D^Y}=\Delta_S$ under Assumption~\ref{ass:structuredS&Ti}.  Similar to Lemma~\ref{lem:|di(MNMFNHess-Hess)di|}, the directional Hessian error bound of GSH in the directions $d^i\slash\|d^i\|$ is $\mathcal{O}(\Delta_S)$.
    \begin{lemma}\label{lem:|di(GSH-Hess)di|}
        Suppose that Assumptions \ref{ass:fC2+}, \ref{ass:structuredS&Ti} hold.  Then, for any non-zero $d^i\in D^Y$,
        \begin{equation}\label{ineq:|di(GSH-Hess)di|}
            \left|\left(\frac{d^i}{\|d^i\|}\right)^\top\left(\nabla_s^2f(x^0;S;T_{1:p})-\nabla^2f(x^0)\right)\frac{d^i}{\|d^i\|}\right| \le \frac{L_{\nabla^2f}}{3}\Delta_S.
        \end{equation}
    \end{lemma}
    \begin{proof}
        Notice that $\nabla_sf(x^0+d^i;T_i) = \left(f(x^0+d^i)-f(x^0)\right)d^i\slash\|d^i\|^2$ and $\nabla_sf(x^0;T_i) = \left(f(x^0)-f(x^0-d^i)\right)d^i\slash\|d^i\|^2$, so
        \begin{align*}
            &~~\nabla_s^2f(x^0;S;T_{1:p})\\
            &= \left(S^\top\right)^\dagger\mathrm{Diag}\left(\left[\tfrac{f(x^0+d^1)+f(x^0-d^1)-2f(x^0)}{\|d^1\|^2}~\cdots~\tfrac{f(x^0+d^p)+f(x^0-d^p)-2f(x^0)}{\|d^p\|^2}\right]\right)S^\top.
        \end{align*}
        By Assumption~\ref{ass:structuredS&Ti}, $S$ has full column rank and so $S^\dagger d^i=e^i$.  Therefore,
        \begin{align*}
            &~~\left|\left(\frac{d^i}{\|d^i\|}\right)^\top\left(\nabla_s^2f(x^0;S;T_{1:p})-\nabla^2f(x^0)\right)\frac{d^i}{\|d^i\|}\right|\\
            &= \frac{1}{\|d^i\|^2}\left|\left(\frac{f(x^0+d^i)+f(x^0-d^i)-2f(x^0)}{\|d^i\|^2}e^i\right)^\top S^\top d^i-\left(d^i\right)^\top\nabla^2f(x^0)d^i\right|\\
            &= \frac{1}{\|d^i\|^2}E(d^i,d^i) \le \frac{L_{\nabla^2f}}{3}\left\|d^i\right\| \le \frac{L_{\nabla^2f}}{3}\Delta_S,
        \end{align*}
        where the last line follows from Lemma~\ref{lem:C2+GivesO(D3)} with $u=d^i$ and $\Delta=\|d^i\|$.
    \end{proof}

    Similar to Lemma~\ref{lem:|di(MNMFNHess-Hess)dj|}, if we replace one of the $d^i\slash\|d^i\|$ on the left-hand side of Inequality \eqref{ineq:|di(GSH-Hess)di|} with~$d^j\slash\|d^j\|$ where $i\neq j$, then the error bound deteriorates to $\mathcal{O}(1)$.  A key difference between Lemmas~\ref{lem:|di(MNMFNHess-Hess)dj|} and \ref{lem:|di(GSH-Hess)dj|} is that the latter does not require the corresponding model to be fully linear.  Consequently, no full-rank assumption on $D^Y$, nor the requirement that $Y$ be poised for MFN interpolation, is needed.
    \begin{lemma}\label{lem:|di(GSH-Hess)dj|}
        Suppose that Assumptions~\ref{ass:fC2+}, \ref{ass:structuredS&Ti} hold. Then, for any non-zero $d^i,d^j\in D^Y$ with $i\neq j$,
        \begin{align*}
            \left|\left(\frac{d^i}{\|d^i\|}\right)^\top\left(\nabla_s^2f(x^0;S;T_{1:p})-\nabla^2f(x^0)\right)\frac{d^j}{\|d^j\|}\right| \le \left\|\nabla^2f(x^0)\right\| + \frac{L_{\nabla^2 f}}{3}\Delta_S.
        \end{align*}
    \end{lemma}
    \begin{proof}
        Using similar proof techniques as in Lemma~\ref{lem:|di(GSH-Hess)di|}, we obtain
        \begin{align*}
            &~~\left|\left(\frac{d^i}{\|d^i\|}\right)^\top\left(\nabla_s^2f(x^0;S;T_{1:p})-\nabla^2f(x^0)\right)\frac{d^j}{\|d^j\|}\right|\\
            &= \frac{1}{\|d^i\|\|d^j\|}\left|\frac{(d^i)^\top d^j}{\|d^i\|^2}\left(f(x^0+d^i)+f(x^0-d^i)-2f(x^0)\right)-\left(d^i\right)^\top\nabla^2f(x^0)d^j\right|\\
            &\le \frac{1}{\|d^i\|^2}E(d^i,d^i)+\frac{1}{\|d^i\|\|d^j\|}\left|\left(d^i\right)^\top \nabla^2 f(x^0)\left(\frac{(d^i)^\top d^j}{\|d^i\|^2}d^i-d^j\right)\right|\\
            &\le \frac{1}{\|d^j\|}\left\|\nabla^2f(x^0)\right\|\left\|\frac{(d^i)^\top d^j}{\|d^i\|^2}d^i-d^j\right\| + \frac{L_{\nabla^2 f}}{3}\Delta_S.
        \end{align*}
        The proof is complete by noticing that $\frac{(d^i)^\top d^j}{\|d^i\|^2}d^i$ is the projection of $d^j$ onto the linear span of $d^i$, so $\|\frac{(d^i)^\top d^j}{\|d^i\|^2}d^i-d^j\| \le \|d^j\|$.
    \end{proof}

    Using Lemmas~\ref{lem:|di(GSH-Hess)di|} and \ref{lem:|di(GSH-Hess)dj|}, we can establish results similar to those in Theorem~\ref{thm:|d(MNMFNHess-Hess)d|}.  However, as we noted before Lemma~\ref{lem:|di(GSH-Hess)dj|}, we do not require any full-rank assumptions on $D^Y$.  Consequently, the directional Hessian error bounds in Theorem~\ref{thm:|d(GSH-Hess)d|} are established for any direction $d\slash\|d\|$ belonging to the linear span of $d^1,\ldots,d^p$.  These results further yield an $\mathcal{O}(1)$ bound on the Hessian approximation error restricted to the linear span of $d^1,\ldots,d^p$.
    \begin{theorem}\label{thm:|d(GSH-Hess)d|}
        Suppose that Assumptions~\ref{ass:fC2+}, \ref{ass:structuredS&Ti} hold. Then, for any non-zero $d\in\mathrm{col}(S)$, we can write $d=Sv$ with $v=S^\dagger d$.  Moreover,
        \begin{align*}
            &~~\left|\left(\frac{d}{\|d\|}\right)^\top\left(\nabla_s^2f(x^0;S;T_{1:p})-\nabla^2f(x^0)\right)\frac{d}{\|d\|}\right|\\
            &\le \frac{\|v\|_1^2-\|v\|_{\infty}^2}{\|v\|^2}\left\|\nabla^2 f(x^0)\right\|\left\|\overline{S}^\dagger\right\|^2 + \frac{L_{\nabla^2 f}}{3}\left\|\overline{S}^\dagger\right\|^2\left(\frac{\|v\|_1^2-\|v\|_{\infty}^2}{\|v\|^2}+1\right)\Delta_S.
        \end{align*}
        In particular, we have that
        \begin{align*}
            &~~\max_{d\in\mathrm{col}(S),\|d\|=1}\left|d^\top\left(\nabla_s^2f(x^0;S;T_{1:p})-\nabla^2f(x^0)\right)d\right|\\
            &\le \left(p-\frac{1}{p}\right)\left\|\nabla^2 f(x^0)\right\|\left\|\overline{S}^\dagger\right\|^2 + \frac{L_{\nabla^2 f}}{3}\left\|\overline{S}^\dagger\right\|^2\left(p-\frac{1}{p}+1\right)\Delta_S.
        \end{align*}
        
    \end{theorem}
    \begin{proof}
        Using similar proof techniques as in Theorem~\ref{thm:|d(MNMFNHess-Hess)d|}, we obtain
        \begin{align*}
            &~~\left|\left(\frac{d}{\|d\|}\right)^\top\left(\nabla_s^2f(x^0;S;T_{1:p})-\nabla^2f(x^0)\right)\frac{d}{\|d\|}\right|\\
            &\le \frac{\|v\|_1^2-\|v\|_{\infty}^2}{\|d\|^2}\left\|\nabla^2 f(x^0)\right\|\Delta_S^2 + \frac{L_{\nabla^2 f}}{3}\frac{\|v\|^2}{\|d\|^2}\Delta_S^3 + \frac{L_{\nabla^2 f}}{3}\frac{\|v\|_1^2-\|v\|_{\infty}^2}{\|d\|^2}\Delta_S^3.
        \end{align*}
        The remaining proof is the same as that of Theorem~\ref{thm:|d(MNMFNHess-Hess)d|}.
    \end{proof}

    \begin{remark}
        Comparing Lemmas \ref{lem:|di(GSH-Hess)di|}, \ref{lem:|di(GSH-Hess)dj|}, and Theorem \ref{thm:|d(GSH-Hess)d|}, we see that the GSH have directional errors ranging from $\mathcal{O}(1)$ to $\mathcal{O}(\Delta_S)$.  Similar to the Hessian approximations arising from the MN and MFN models, it achieves the same level of accuracy as the Hessian approximation of fully quadratic models along the directions $d^i\slash\|d^i\|$.
    \end{remark}

\section{Relationships among MN, MFN, and QS models}\label{sec:relationships}
    In this section, we explore the relationships among the MN, MFN, and QS models. In particular, we examine the conditions under which they coincide. These results not only deepen our understanding of these models but also provide an alternative, potentially simpler approach to constructing them. 

\subsection{Relationships among MN, MFN, and QS models in the case $T_{1:p}=T$}
    In this subsection, we suppose that all $T_i=T=[t^1\cdots t^q]$ and
    \begin{align*}
            Y=\left\{x^0,x^0+s^i,x^0+t^j,x^0+s^i+t^j:i\in [p],j\in [q]\right\}.
        \end{align*}
    As mentioned in Remark \ref{rem:GSH(S,T)^T=GSH(T,S)}, we have $\nabla_s^2f(x^0;S;T)^\top=\nabla_s^2f(x^0;T;S)$.  Therefore, in this subsection we present results only for $\nabla_s^2f(x^0;S;T)$.  Similar results can be established for $\nabla_s^2f(x^0;T;S)$ by interchanging the roles of $S$ and $T$.
    
    Direct row operations can show that the equality constraints of Problems~\eqref{pro:MN} and~\eqref{pro:MFN} are equivalent to
    \begin{align}\label{constraints:MN/MFN_QS}
        \begin{cases}
            &\left(s^i\right)^\top\alpha+\frac{1}{2}\left(s^i\right)^\top H\left(s^i\right) = f(x^0+s^i)-f(x^0),~~~i\in [p],\\
            &\left(t^j\right)^\top\alpha+\frac{1}{2}\left(t^j\right)^\top H\left(t^j\right) = f(x^0+t^j)-f(x^0),~~~j\in [q],\\
            &S^\top HT=\delta_{\delta_f}(x^0;S;T).
        \end{cases}
    \end{align}
    Consider the following problem that only involves $H$ as a variable:
    \begin{equation}\label{pro:minH_siHsj_allsymmat}
        \min_{H\in \mathbb{R}^{n\times n}}~~~\frac{1}{2}\left\|H\right\|_F^2~~~s.t.~~~S^\top HT=\delta_{\delta_f}(x^0;S;T),~H\in S_n(\mathbb{R}).
    \end{equation}
    Clearly, for all $(\alpha,H)$ feasible for System \eqref{constraints:MN/MFN_QS}, $H$ is feasible for Problem~\eqref{pro:minH_siHsj_allsymmat}.  Therefore, the solution to Problem~\eqref{pro:minH_siHsj_allsymmat} should be closely related to the solution to Problems~\eqref{pro:MN} and \eqref{pro:MFN}.

    We begin by introducing several lemmas that examine the properties of $\nabla_s^2f(x^0;S;T)$.  In particular, we discuss when it is the solution to Problem~\eqref{pro:minH_siHsj_allsymmat}.  As will be shown later, these results are crucial to linking the MN, MFN, and QS models.  The next lemma shows that $\nabla_s^2f(x^0;S;T)$ is always the unique solution to Problem~\eqref{pro:minH_siHsj_allsymmat} when the symmetry requirement is removed.
    \begin{lemma}\label{lem:minH_siHsj}
        Suppose that Problem~\eqref{pro:minH_siHsj_allmat} is feasible.  Then, $\nabla_s^2f(x^0;S;T)$ is the unique solution to
        \begin{equation}\label{pro:minH_siHsj_allmat}
            \min_{H\in \mathbb{R}^{n\times n}}~~~\frac{1}{2}\left\|H\right\|_F^2~~~s.t.~~~S^\top HT=\delta_{\delta_f}(x^0;S;T).
        \end{equation}
        In particular, $\nabla_s^2f(x^0;S;T)$ is the unique solution to Problem~\eqref{pro:minH_siHsj_allsymmat} if and only if Problem~\eqref{pro:minH_siHsj_allsymmat} is feasible and $\nabla_s^2f(x^0;S;T)\in S_n(\mathbb{R})$ (or equivalently, $\nabla_s^2f(x^0;S;T)$ is feasible for Problem~\eqref{pro:minH_siHsj_allsymmat}).
    \end{lemma}
    \begin{proof}
        Notice that the objective functions of Problems~\eqref{pro:minH_siHsj_allsymmat} and \eqref{pro:minH_siHsj_allmat}  are strongly convex, and all constraints are linear with respect to $H$.  The solution exists and is unique if and only if the problem is feasible. Moreover, the KKT conditions are necessary and sufficient for optimality. 
        
        Let $\Lambda\in\mathbb{R}^{q\times p}$ be the Lagrange multiplier.  The Lagrangian function of Problem~\eqref{pro:minH_siHsj_allmat} is $L(H,\Lambda) = \frac{1}{2}\|H\|_F^2 + \mathrm{tr}((S^\top HT-\delta_{\delta_f}(x^0;S;T))\Lambda)$.  The KKT conditions are $H + S\Lambda^\top T^\top = \mymathbb{0}_{n\times n}$ and $S^\top HT - \delta_{\delta_f}(x^0;S;T) = \mymathbb{0}_{p\times q}$.
        
        We claim that  $(H^*,\Lambda^*)=(\nabla_s^2f(x^0;S;T),-(S^\dagger H^*(T^\top)^\dagger)^\top)$ satisfies the KKT conditions.  Indeed, since $S^\top HT=\delta_{\delta_f}(x^0;S;T)$ is feasible, every column of $\delta_{\delta_f}(x^0;S;T)$ lies in $\mathrm{col}(S^\top)$ and every row of $\delta_{\delta_f}(x^0;S;T)$ lies in $\mathrm{col}(T^\top)$, i.e., $S^\top\left(S^\top\right)^\dagger\delta_{\delta_f}(x^0;S;T) = \delta_{\delta_f}(x^0;S;T)$ and $\delta_{\delta_f}(x^0;S;T)T^\dagger T  = \delta_{\delta_f}(x^0;S;T)$.  Moreover, every column of $H^*$ lies in $\mathrm{col}((S^\top)^\dagger)=\mathrm{col}(S)$ and every row of $H^*$ lies in $\mathrm{col}((T^\dagger)^\top)=\mathrm{col}(T)$, i.e., $S S^\dagger H^* = H^*$ and $H^*\left(T^\top\right)^\dagger T^\top = H^*$.  Therefore, $S^\top H^*T - \delta_{\delta_f}(x^0;S;T) = S^\top(S^\top)^\dagger\delta_{\delta_f}(x^0;S;T)T^\dagger T - \delta_{\delta_f}(x^0;S;T) = \mymathbb{0}_{p\times q}$ and $H^* + S(\Lambda^*)^\top T^\top = H^* - SS^\dagger H^*(T^\top)^\dagger T^\top = \mymathbb{0}_{n\times n}$, i.e., $H^*=\nabla_s^2f(x^0;S;T)$ is a KKT point and therefore the unique solution.  It immediately follows that $\nabla_s^2f(x^0;S;T)$ is the unique solution to Problem~\eqref{pro:minH_siHsj_allsymmat} if and only if Problem~\eqref{pro:minH_siHsj_allsymmat} is feasible and $\nabla_s^2f(x^0;S;T)\in S_n(\mathbb{R})$. 
    \end{proof}

    Lemma~\ref{lem:minH_siHsj} provides a necessary and sufficient condition for $\nabla_s^2f(x^0;S;T)$ to be the unique solution to Problem~\eqref{pro:minH_siHsj_allsymmat}. Now, we provide a detailed characterization of the solution to Problem~\eqref{pro:minH_siHsj_allsymmat}, under an extra assumption on $\mathrm{col}(T)$ and $\mathrm{col}(S)$. 
    \begin{lemma}\label{lem:solutionTOminH_siHsj_allsymmat_colTincolS}
        Suppose that Problem~\eqref{pro:minH_siHsj_allsymmat} is feasible and $\mathrm{col}(T)\subseteq\mathrm{col}(S)$.  The unique solution to Problem~\eqref{pro:minH_siHsj_allsymmat} can be expressed, and in fact must take the form $\nabla_s^2f(x^0;S;T) + (S^\top)^\dagger Z\left(I_n-TT^\dagger\right)\in S_n(\mathbb{R})$, for some $Z\in\mathbb{R}^{p\times n}$.  In particular, if $T$ has full row rank, then $\nabla_s^2f(x^0;S;T)$ is the unique solution to Problem~\eqref{pro:minH_siHsj_allsymmat}.
    \end{lemma}
    \begin{proof}
        As mentioned in the proof of Lemma~\ref{lem:minH_siHsj}, the solution to Problem~\eqref{pro:minH_siHsj_allsymmat} exists and is unique if and only if the problem is feasible.  Moreover, the KKT conditions are necessary and sufficient conditions for optimality.
        
        Let $\Lambda_1\in\mathbb{R}^{q\times p}$ and $\Lambda_2\in\mathbb{R}^{n\times n}$.  The Lagrangian function is 
        \begin{equation*}
            L(H,\Lambda_1,\Lambda_2) = \frac{1}{2}\left\|H\right\|_F^2 + \mathrm{tr}\left(\left(S^\top HT-\delta_{\delta_f}(x^0;S;T)\right)\Lambda_1\right) + \mathrm{tr}\left(\left(H- H^\top\right)\Lambda_2\right).
        \end{equation*}
        The KKT conditions are $H + S\Lambda_1^\top T^\top + \Lambda_2^\top - \Lambda_2 = \mymathbb{0}_{n\times n}$, $S^\top HT=\delta_{\delta_f}(x^0;S;T)$, and $H=H^\top$.  Since $H$ is symmetric, we separate the symmetric and skew parts from the first KKT condition and get \begin{align}
            &H = -\frac{1}{2}\left(S\Lambda_1^\top T^\top + T\Lambda_1S^\top\right)\label{eq:H*_allsymmat}\\
            &\frac{1}{2}\left(S\Lambda_1^\top T^\top - T\Lambda_1S^\top\right) = \Lambda_2 - \Lambda_2^\top.\label{eq:Lambda2*_allsymmat}
        \end{align}
        The second KKT condition gives $-\frac{1}{2}S^\top(S\Lambda_1^\top T^\top + T\Lambda_1S^\top)T = \delta_{\delta_f}(x^0;S;T)$, i.e., $\Lambda_1$ is the solution to
        \begin{equation}\label{eq:Lambda1*_allsymmat}
            -\frac{1}{2}\left(S^\top S\Lambda_1^\top T^\top T + S^\top T\Lambda_1 S^\top T\right) = \delta_{\delta_f}(x^0;S;T).
        \end{equation}
        The KKT conditions of Problem~\eqref{pro:minH_siHsj_allsymmat} are equivalent to Equations~\eqref{eq:H*_allsymmat}, \eqref{eq:Lambda2*_allsymmat}, and  \eqref{eq:Lambda1*_allsymmat}.  Equation~\eqref{eq:Lambda2*_allsymmat} is always feasible by letting $\Lambda_2=\frac{1}{2}S\Lambda_1^\top T^\top$.

        Since $\mathrm{col}(T)\subseteq\mathrm{col}(S)$, we have $T=SM$ for some matrix $M\in\mathbb{R}^{p\times q}$.  Then, Equation~\eqref{eq:H*_allsymmat} implies that all columns of $H$ lie in $\mathrm{col}(S)=\mathrm{col}((S^\top)^\dagger)$, i.e., $(S^\top)^\dagger S^\top H=H$.  Moreover, Equation~\eqref{eq:Lambda1*_allsymmat} becomes
        \begin{equation*}
            -\frac{1}{2}S^\top S\left(\Lambda_1^\top M^\top + M\Lambda_1\right)S^\top T = \delta_{\delta_f}(x^0;S;T),
        \end{equation*}
        with the general solution in the form
        \begin{equation*}
            -\frac{1}{2}S^\top S\left(\Lambda_1^\top M^\top + M\Lambda_1\right)S^\top = \delta_{\delta_f}(x^0;S;T)T^\dagger + Z\left(I_n-TT^\dagger\right),
        \end{equation*}
        where $Z\in\mathbb{R}^{p\times n}$.  Since the solution to Problem~\eqref{pro:minH_siHsj_allsymmat} exists, there must exist $Z$ and $\Lambda_1$ such that the equation above is feasible and so Equation~\eqref{eq:H*_allsymmat} gives
        \begin{align*}
            H = \left(S^\top\right)^\dagger S^\top H &= -\frac{1}{2}\left(S^\top\right)^\dagger S^\top S\left(\Lambda_1^\top M^\top + M\Lambda_1\right)S^\top\\
            &= \nabla_s^2f(x^0;S;T) + (S^\top)^\dagger Z\left(I_n-TT^\dagger\right),
        \end{align*}
        i.e., $H=\nabla_s^2f(x^0;S;T) + (S^\top)^\dagger Z\left(I_n-TT^\dagger\right)$ is a KKT point and thus is the unique solution.  The proof is complete by noticing that if $T$ has full row rank, then $I_n-TT^\dagger=\mymathbb{0}_{n\times n}$. 
    \end{proof}
    
    The special case in which $\nabla_s^2 f(x^0; S; T)$ is the unique solution to Problem~\eqref{pro:minH_siHsj_allsymmat}, as given in Lemma~\ref{lem:solutionTOminH_siHsj_allsymmat_colTincolS}, in fact imply that $\mathrm{col}(S)=\mathrm{col}(T)$.  More generally, we can show that if Problem~\eqref{pro:minH_siHsj_allsymmat} is feasible and $\mathrm{col}(S)=\mathrm{col}(T)$, then $\nabla_s^2 f(x^0; S; T)$ is the unique solution to Problem~\eqref{pro:minH_siHsj_allsymmat}, i.e., the result holds without any full rank assumption on $T$.
    \begin{corollary}\label{cor:GSHisH*_minH_siHsj_allsymmat}
        Suppose Problem~\eqref{pro:minH_siHsj_allsymmat} is feasible and $\mathrm{col}(T)=\mathrm{col}(S)$.  Then, $\nabla_s^2f(x^0;S;T)$ must be the unique solution to Problem~\eqref{pro:minH_siHsj_allsymmat}. 
    \end{corollary}
    \begin{proof}
        Notice that we have $\mathrm{col}(T)\subseteq\mathrm{col}(S)$. By Lemma~\ref{lem:solutionTOminH_siHsj_allsymmat_colTincolS}, the unique solution to Problem~\eqref{pro:minH_siHsj_allsymmat} has the form $H=\nabla_s^2f(x^0;S;T) + (S^\top)^\dagger Z\left(I_n-TT^\dagger\right)$ for some $Z\in\mathbb{R}^{p\times n}$.  Now, since $\mathrm{col}(T)=\mathrm{col}(S)$, Equation~\eqref{eq:H*_allsymmat} implies that $HTT^\dagger=H$.  Using the identity $T^{\dagger} T T^{\dagger}=T^{\dagger}$, we obtain
        \begin{align*}
            H = \nabla_s^2f(x^0;S;T)TT^\dagger + (S^\top)^\dagger Z\left(I_n-TT^\dagger\right)TT^\dagger = \nabla_s^2f(x^0;S;T).
        \end{align*}
    \end{proof}

    The results of Lemmas~\ref{lem:minH_siHsj}, \ref{lem:solutionTOminH_siHsj_allsymmat_colTincolS}, and Corollary \ref{cor:GSHisH*_minH_siHsj_allsymmat} are shown in Table \ref{tab:sols_minH_siHsj_allmat&allsymmat}.
    \begin{table}[htbp]
    \centering
    \caption{Summary of the unique solution to Problems~\eqref{pro:minH_siHsj_allsymmat} and \eqref{pro:minH_siHsj_allmat}} 
    \label{tab:sols_minH_siHsj_allmat&allsymmat}
    \begin{tabular}{c|c}
    \hline
        Assumption &  Unique Solution\\ \hline
        \begin{tabular}{c}
            \eqref{pro:minH_siHsj_allsymmat} is feasible and \\
            $\nabla_s^2f(x^0;S;T)\in S_n(\mathbb{R})$
        \end{tabular} & $\nabla_s^2f(x^0;S;T)$ \\ \hline
        \begin{tabular}{c}
            \eqref{pro:minH_siHsj_allsymmat} is feasible and \\
            $\mathrm{col}(T)\subseteq\mathrm{col}(S)$
        \end{tabular} & $\nabla_s^2f(x^0;S;T) + \left(S^\top\right)^\dagger Z\left(I_n-TT^\dagger\right)$ \\ \hline
        \begin{tabular}{c}
            \eqref{pro:minH_siHsj_allsymmat} is feasible and \\
            $\mathrm{col}(T)=\mathrm{col}(S)$
        \end{tabular} & $\nabla_s^2f(x^0;S;T)$ \\ \hline
        \eqref{pro:minH_siHsj_allmat} is feasible & $\nabla_s^2f(x^0;S;T)$ \\ \hline
    \end{tabular}
    \end{table}

    Now, we discuss some special structures of $T$ that make $\nabla_s^2f(x^0;S;T)$ feasible for Problem~\eqref{pro:minH_siHsj_allsymmat}, which implies that $\nabla_s^2f(x^0;S;T)$ is the unique solution to Problem~\eqref{pro:minH_siHsj_allsymmat}. In fact, if $T$ has these structures, then the feasibility is independent of $f$.
    \begin{example}\label{exp:symdeltadeltafMinv_specialT}
        Suppose that $S$ has full column rank.  Let $\ell\in\{0,1,\ldots,p\}$.  Define $U_S^0:=S$ and $U_S^\ell:=[s^1-s^\ell\cdots s^{\ell-1}-s^\ell~~-s^\ell~~s^{\ell+1}-s^\ell\cdots s^p-s^\ell]$ for $\ell>0$.  If $T=U_S^\ell$, then for all functions $f:\mathbb{R}^n\to\mathbb{R}$, $\nabla_s^2f(x^0;S;T)$ is the unique solution to Problem~\eqref{pro:minH_siHsj_allsymmat}.
    \end{example}
    \begin{proof}
        Combining \cite[Prop.~5.7]{hare2024matrix} with \cite[Thm.~5]{chen2026relationships}, we obtain that there exists $\alpha\in\mathbb{R}^n$ such that $(\alpha,\nabla_s^2f(x^0;S;T))$ is feasible for System \eqref{constraints:MN/MFN_QS}.  Therefore, $\nabla_s^2f(x^0;S;T)$ is feasible for Problem~\eqref{pro:minH_siHsj_allsymmat}, and so is the unique solution.  
    \end{proof}

    The next lemma, inspired by \cite[Prop.~5.5]{hare2024matrix}, discusses how certain transformations on $S$ and $T$ influence the symmetry of $\nabla_s^2f(x^0;S;T)$ and its feasibility for Problem~\eqref{pro:minH_siHsj_allsymmat}.
    \begin{lemma}\label{lem:symdeltadeltafMinv_NSP&NTP}
        Let $N\in\mathbb{R}^{n\times n}$ be any orthogonal matrix.  Let $P_1\in\mathbb{R}^{p\times p}$ and $P_2\in\mathbb{R}^{q\times q}$ be any permutation matrices. Define $\widetilde{S}:=NSP_1$, $\widetilde{T}:=NTP_2$, and $\widetilde{f}(x):=f(x^0+N^{-1}(x-x^0))$.  If $\nabla_s^2f(x^0;S;T)$ is symmetric, then $\nabla_s^2\widetilde{f}(x^0;\widetilde{S};\widetilde{T})$ is symmetric.  Moreover, if $\nabla_s^2f(x^0;S;T)$ is feasible for Problem~\eqref{pro:minH_siHsj_allsymmat}, then $\nabla_s^2\widetilde{f}(x^0;\widetilde{S};\widetilde{T})$ is feasible for Problem~\eqref{pro:minH_siHsj_allsymmat} defined on $\widetilde{S}$, $\widetilde{T}$, and $\widetilde{f}$.  
    \end{lemma}
    \begin{proof}
        Suppose that $\nabla_s^2f(x^0;S;T)=(S^\top)^\dagger\delta_{\delta_f}(x^0;S;T)T^\dagger$ is symmetric.  By definition, we have that $\delta_{\delta_{\widetilde{f}}}(x^0;\widetilde{S};\widetilde{T})=\delta_{\delta_f}(x^0;SP_1;TP_2)$.  Suppose that $SP_1=[s^{k_1}\cdots s^{k_p}]$ and $TP_2=[t^{l_1}\cdots t^{l_q}]$.  Then, direction computation can show that $(\delta_{\delta_f}(x^0;SP_1;TP_2))_{ij}=(\delta_{\delta_f}(x^0;S;T))_{k_il_j}$.  This means we have $\delta_{\delta_f}(x^0;SP_1;TP_2)=P_1^\top\delta_{\delta_f}(x^0;S;T)P_2$ and so
        \begin{align*}
            \nabla_s^2\widetilde{f}(x^0;\widetilde{S};\widetilde{T}) = \left(\widetilde{S}^\top\right)^\dagger\delta_{\delta_f}(x^0;SP_1;TP_2)\widetilde{T}^\dagger = N\left(S^\top\right)^\dagger \delta_{\delta_f}(x^0;S;T)T^\dagger N^\top,
        \end{align*}
        is symmetric.  Now, suppose that $\nabla_s^2f(x^0;S;T)$ is feasible for Problem~\eqref{pro:minH_siHsj_allsymmat}.  Notice that $\widetilde{S}^\top(\widetilde{S}^\top)^\dagger=P_1^\top S^\top N^{-1}N(S^\top)^\dagger P_1=P_1^\top S^\top(S^\top)^\dagger P_1$ and, similarly,  $\widetilde{T}^\dagger\widetilde{T}=P_2^\top T^\dagger TP_2$.  We obtain
        \begin{align*}
            &~~\widetilde{S}^\top\nabla_s^2\widetilde{f}(x^0;\widetilde{S};\widetilde{T})\widetilde{T}\\
            &= P_1^\top S^\top\left(S^\top\right)^\dagger P_1\delta_{\delta_f}(x^0;SP_1;TP_2)P_2^\top T^\dagger TP_2\\
            &= P_1^\top S^\top\left(S^\top\right)^\dagger \delta_{\delta_f}(x^0;S;T)T^\dagger TP_2 = P_1^\top S^\top\nabla_s^2f(x^0;S;T) TP_2\\
            &= P_1^\top\delta_{\delta_f}(x^0;S;T)P_2 = \delta_{\delta_f}(x^0;SP_1;TP_2) = \delta_{\delta_{\widetilde{f}}}(x^0;\widetilde{S};\widetilde{T}),
        \end{align*}
        and so $\nabla_s^2\widetilde{f}(x^0;\widetilde{S};\widetilde{T})$ is feasible for Problem~\eqref{pro:minH_siHsj_allsymmat} defined on $\widetilde{S}$, $\widetilde{T}$, and $\widetilde{f}$.  
    \end{proof}

\subsubsection{Constructing MN and MFN models from GSG and GSH in the case $\mathrm{col}(T)=\mathrm{col}(S)$}\label{subsubsec:QS=MN&MFN_colT=colS}
    In this subsection, we specifically consider the case where $\mathrm{col}(T)=\mathrm{col}(S)$.  We have proved that in this case, if Problem~\eqref{pro:minH_siHsj_allsymmat} is feasible, then $\nabla^2f(x^0;S;T)$ is the unique solution to Problem~\eqref{pro:minH_siHsj_allsymmat}.  Based on this, we provide exact formulas to construct MN and MFN models using the GSG and GSH.  We also discuss when the MN, MFN, and QS models coincide.
    
    The following theorem characterizes the relationship among the GSG, GSH, and MN models. In particular, $\nabla_{\mathrm{MN}}^2 f(x^0; Y)$ coincides exactly with the GSH of $f$ at $x^0$ over $S$ and $T$, while $\nabla_{\mathrm{MN}} f(x^0; Y)$ corresponds to the GSG of $f$ at $x^0$ over $S$, augmented by an additional correction term.
    \begin{theorem}\label{thm:QS=MN_colT=colS}
        Suppose Problem~\eqref{pro:MN} is feasible and $\mathrm{col}(T)=\mathrm{col}(S)$. Then, we have $\nabla_{\mathrm{MN}}f(x^0;Y) = \nabla_s f(x^0;S)-\frac{1}{2}(S^\top)^\dagger\mathrm{diag}(\delta_{\delta_f}(x^0;S;T)T^\dagger S)$ and $\nabla_{\mathrm{MN}}^2f(x^0;Y) = \nabla_s^2f(x^0;S;T)$.
    \end{theorem}
    \begin{proof}
        Since $\mathrm{col}(S)=\mathrm{col}(T)$, there exist $M_1\in\mathbb{R}^{p\times q}$ and $M_2\in\mathbb{R}^{q\times p}$ such that $T=SM_1$ and $S=TM_2$.  We take $M_1=S^\dagger T$ and $M_2=T^\dagger S$. The last equality in System \eqref{constraints:MN/MFN_QS} is $S^\top HT=S^\top HSM_1=M_2^\top T^\top HT=\delta_{\delta_f}(x^0;S;T)$, which gives
        \begin{align*}
            S^\top HS &= M_2^\top T^\top HTM_2=\delta_{\delta_f}(x^0;S;T)M_2=\delta_{\delta_f}(x^0;S;T)T^\dagger S\\
            T^\top HT &= M_1^\top S^\top HSM_1=M_1^\top\delta_{\delta_f}(x^0;S;T)=T^\top\left(S^\top\right)^\dagger\delta_{\delta_f}(x^0;S;T).
        \end{align*}
        Using these equations to cancel the term containing $H$ in the remaining equalities in System \eqref{constraints:MN/MFN_QS}, we obtain that System \eqref{constraints:MN/MFN_QS} is equivalent to
        \begin{align*}
        \begin{cases}
            &S^\top\alpha = \delta_f(x^0;S) - \frac{1}{2}\mathrm{diag}\left(\delta_{\delta_f}(x^0;S;T)T^\dagger S\right),\\
            &T^\top\alpha = \delta_f(x^0;T) - \frac{1}{2}\mathrm{diag}\left(T^\top\left(S^\top\right)^\dagger\delta_{\delta_f}(x^0;S;T)\right),\\
            &S^\top HT = \delta_{\delta_f}(x^0;S;T).
        \end{cases}
        \end{align*}
        Since $\mathrm{col}(S)=\mathrm{col}(T)$ and  System \eqref{constraints:MN/MFN_QS} is feasible, the first two constraints in the system above must be equivalent, and Problem~\eqref{pro:minH_siHsj_allsymmat} is feasible.  Thus, $\nabla_{\mathrm{MN}}^2f(x^0;Y)$ is the solution to Problem~\eqref{pro:minH_siHsj_allsymmat} and $\nabla_{\mathrm{MN}}f(x^0;Y)$ is the minimum norm solution to $S^\top\alpha = \delta_f(x^0;S) - \frac{1}{2}\mathrm{diag}(\delta_{\delta_f}(x^0;S;T)T^\dagger S)$. Therefore, we have $\nabla_{\mathrm{MN}}^2f(x^0;Y) = \nabla_s^2f(x^0;S;T)$ by Corollary \ref{cor:GSHisH*_minH_siHsj_allsymmat}, and
        \begin{align*}
            \nabla_{\mathrm{MN}}f(x^0;Y) &= \left(S^\top\right)^\dagger\delta_f(x^0;S)-\frac{1}{2}\left(S^\top\right)^\dagger\mathrm{diag}\left(\delta_{\delta_f}(x^0;S;T)T^\dagger S\right)\\
            &= \nabla_s f(x^0;S)-\frac{1}{2}\left(S^\top\right)^\dagger\mathrm{diag}\left(\delta_{\delta_f}(x^0;S;T)T^\dagger S\right).
        \end{align*}
    \end{proof}

    Following the same proof as that of Theorem~\ref{thm:QS=MN_colT=colS}, we obtain similar results for the MFN model.  A key distinction, however, is that Corollary \ref{cor:GSHisH*_minH_siHsj_allsymmat} guarantees only the uniqueness of $\nabla_{\mathrm{MFN}}^2 f(x^0; Y)$.  In contrast, the uniqueness of $\nabla_{\mathrm{MFN}} f(x^0; Y)$ requires an additional assumption, namely that $Y$ is poised for MFN interpolation.

    \begin{theorem}\label{thm:QS=MFN_colT=colS}
        Suppose Problem~\eqref{pro:MFN} is feasible and $\mathrm{col}(T)=\mathrm{col}(S)$. Then, $\nabla_{\mathrm{MFN}}^2f(x^0;Y) = \nabla_s^2f(x^0;S;T)$ and $\nabla_{\mathrm{MFN}}f(x^0;Y)$ is any vector such that $(\nabla_{\mathrm{MFN}}f(x^0;Y),\nabla_s^2f(x^0;S;T))$ is feasible.  If $Y$ is poised for MFN interpolation, then $\nabla_{\mathrm{MFN}}f(x^0;Y) = \nabla_s f(x^0;S)-\frac{1}{2}(S^\top)^\dagger\mathrm{diag}(\delta_{\delta_f}(x^0;S;T)T^\dagger S)$.
    \end{theorem}

    \begin{remark}
        Comparing Theorems \ref{thm:QS=MN_colT=colS} and \ref{thm:QS=MFN_colT=colS}, we observe that the MN and MFN models coincide when $\mathrm{col}(T)=\mathrm{col}(S)$ and $Y$ is poised for MFN interpolation.  As we can see from the proof of Theorem \ref{thm:QS=MN_colT=colS}, this ultimately stems from the fact that this particular structure of $Y$ allows Problem \eqref{pro:MN} to be separable in $(\alpha,H)$.
    \end{remark}

    An interesting special case is where $S$ has full column rank and $T$ has the structure discussed in Example \ref{exp:symdeltadeltafMinv_specialT}.  The proof of Example \ref{exp:symdeltadeltafMinv_specialT} shows that System \eqref{constraints:MN/MFN_QS} is feasible.  Moreover, $\nabla_{\mathrm{MN}}f(x^0;Y)$ and $\nabla_{\mathrm{MN}}^2f(x^0;Y)$ are linear combinations of GSG and GSH.  In particular, $m^{x^0,Y}_{\mathrm{MN}}(x)$ belongs to the class of QS models of $f$ at $x^0$ over $Y$.
    \begin{corollary}\label{cor:MN=QS_specialT}
        Suppose $S$ has full column rank and $T$ is as in Example~\ref{exp:symdeltadeltafMinv_specialT}.  Then,  $\nabla_{\mathrm{MN}}f(x^0;Y) = \nabla_s f(x^0;S) + \nabla_s f(x^0-s^\ell;S) -\nabla_s f(x^0-s^\ell;2S)$ and $\nabla_{\mathrm{MN}}^2f(x^0;Y) = \nabla_s^2f(x^0;S;T)$.  In particular, $m^{x^0,Y}_{\mathrm{MN}}(x)$ belongs to the class of QS models of $f$ at $x^0$ over $Y$.
    \end{corollary}
    \begin{proof}
        Let $T=U_S^\ell$.  Notice that $\mathrm{col}(S)=\mathrm{col}(T)$.  For any vector $v$, we define $v^0:=\mymathbb{0}$.  Define $M^\ell:=I_p-e^\ell\mymathbb{1}^\top-e^\ell(e^\ell)^\top$.  Then, $T=SM^\ell$.  Since $(M^\ell)^{-1}=M^\ell$ and $S$ has full column rank, $T^\dagger S=M^\ell S^\dagger S=M^\ell$.  Thus, the $i$-th element of $\mathrm{diag}(\delta_{\delta_f}(x^0;S;T)T^\dagger S)$ is $f(x^0+2s^i-s^\ell)-2f(x^0+s^i-s^\ell)+f(x^0-s^\ell)$ and so $\mathrm{diag}(\delta_{\delta_f}(x^0;S;T)T^\dagger S) = \delta_f(x^0-s^\ell;2S)-2\delta_f(x^0-s^\ell;S)$.  Therefore, 
        \begin{align*}
            \nabla_{\mathrm{MN}}f(x^0;Y) &= \nabla_s f(x^0;S)-\frac{1}{2}\left(S^\top\right)^\dagger\left(\delta_f(x^0-s^\ell;2S)-2\delta_f(x^0-s^\ell;S)\right)\\
            &= \nabla_s f(x^0;S) + \left(S^\top\right)^\dagger\delta_f(x^0-s^\ell;S) - \left(2S^\top\right)^\dagger\delta_f(x^0-s^\ell;2S)\\
            &= \nabla_s f(x^0;S) + \nabla_s f(x^0-s^\ell;S) -\nabla_s f(x^0-s^\ell;2S).
        \end{align*}
    \end{proof}
    \begin{remark}\label{rem:specialMNgrad=GACSG}
        When $\ell=0$, $\nabla_{\mathrm{MN}}f(x^0;Y)=2\nabla_s f(x^0;S)-\nabla_s f(x^0;2S)$ and $\nabla_{\mathrm{MN}}^2f(x^0;Y)=\nabla_s^2f(x^0;S;S)$.  This extends the result in \cite[Thm.~1]{chen2024qfully}, where it was only proved that the corresponding model satisfies interpolation conditions. Moreover, $2\nabla_s f(x^0;S)-\nabla_s f(x^0;2S)$ is the \emph{generalized adapted centred simplex gradient} over $\mathbb{Y}=\mathbb{Y}^+\cup\widetilde{\mathbb{Y}}$, where $\mathbb{Y}^+=\{x^0,x^0+s^i:i\in [p]\}$ and $\widetilde{\mathbb{Y}}=\{x^0,x^0+2s^i:i\in [p]\}$.  Detailed explanations and more properties of the generalized adapted centred simplex gradient can be found in \cite{chen2023adapting,chen2025general}.
    \end{remark}

    Similarly, we obtain corresponding results for the MFN model.  Once again, an additional assumption that $Y$ is poised for MFN interpolation is required to guarantee the uniqueness of $\nabla_{\mathrm{MFN}} f(x^0; Y)$.
    \begin{corollary}\label{cor:MFN=QS_specialT}
        Suppose $S$ has full column rank and $T$ is as in Example~\ref{exp:symdeltadeltafMinv_specialT}. Then, $\nabla_{\mathrm{MFN}}^2f(x^0;Y) = \nabla_s^2f(x^0;S;T)$.  If $Y$ is poised for MFN interpolation, then $\nabla_{\mathrm{MFN}}f(x^0;Y) = \nabla_s f(x^0;S) + \nabla_s f(x^0-s^\ell;S) -\nabla_s f(x^0-s^\ell;2S)$ and $m^{x^0,Y}_{\mathrm{MFN}}(x)$ belongs to the class of QS models of $f$ at $x^0$ over $Y$.
    \end{corollary}

    \begin{remark}
        Comparing Corollaries \ref{cor:MN=QS_specialT} and \ref{cor:MFN=QS_specialT}, we observe that the MN, MFN, and QS models coincide when $S$ has full column rank, $T$ has the structure given in Example \ref{exp:symdeltadeltafMinv_specialT}, and $Y$ is poised for MFN interpolation.
    \end{remark}

\subsection{Relationships among MN, MFN, and QS models in the case of different $T_i$}
    In this subsection, we suppose that all $T_i$ are different.  We show that, when $S$ and $T_i$ have a specific structure, results similar to these in Subsection \ref{subsubsec:QS=MN&MFN_colT=colS} exist.
    
    The following theorem is analogous to Theorem~\ref{thm:QS=MN_colT=colS}. Given the structures of $S$ and $T_i$, the MN model can be constructed using GSG and GSH. In particular, it belongs to the class of QS models of $f$ at $x^0$ over $Y$. Notably, these structures alone are sufficient to establish the results.  No assumptions on the column spaces are required.
    \begin{theorem}\label{thm:QS=MN_diffTi}
        Suppose that $S=[e^1\cdots e^p]$ and $T_i=[-e^i]$, $i\in [p]$, where $p\le n$. Then, we have $\nabla_{\mathrm{MN}}f(x^0;Y) = \frac{1}{2}\left(\nabla_s f(x^0;S)+\nabla_s f(x^0;-S)\right)$ and $\nabla_{\mathrm{MN}}^2f(x^0;Y) = \nabla_s^2f(x^0;S;T_{1:p})$.  In particular, $m^{x^0,Y}_{\mathrm{MN}}(x)$ belongs to the class of QS models of $f$ at $x^0$ over $Y$.
    \end{theorem}
    \begin{proof}
        Note that $Y=\{x^0,x^0+e^i,x^0-e^i:i\in [p]\}$.  Direct row operations show that the equality constraints of Problem~\eqref{pro:MN} are equivalent to
        \begin{align*}
        \begin{cases}
            &\left(e^i\right)^\top\alpha = \frac{1}{2}\left(f(x^0+e^i)-       f(x^0)\right)-\frac{1}{2}\left(f(x^0-e^i)-       f(x^0)\right),~~~i\in [p],\\
            &\left(e^i\right)^\top H\left(e^i\right) = f(x^0+e^i) + f(x^0-e^i)-2f(x^0),~~~i\in [p].
        \end{cases}
        \end{align*} 
        Writing the first set of equalities as $S^\top\alpha=\frac{1}{2}(\delta_f (x^0;S)-\delta_f(x^0;-S))$, we get
        \begin{align*}
            \nabla_{\mathrm{MN}}f(x^0;Y) &= \frac{1}{2}\left(\nabla_s f(x^0;S)+\nabla_s f(x^0;-S)\right)\\
            \nabla_{\mathrm{MN}}^2f(x^0;Y) &= \mathrm{Diag}\left(\begin{bmatrix}
                f(x^0+e^1) + f(x^0-e^1)-2f(x^0)\\
                \vdots\\
                f(x^0+e^p) + f(x^0-e^p)-2f(x^0)\\
                \mymathbb{0}
            \end{bmatrix}\right) \in \mathbb{R}^{n\times n}.
        \end{align*}
        It is straightforward to check that $\nabla_{\mathrm{MN}}^2f(x^0;Y)=\nabla_s^2f(x^0;S;T_{1:p})$.
    \end{proof}
    \begin{remark}
        The $\frac{1}{2}(\nabla_s f(x^0;S)+\nabla_s f(x^0;-S))$ is the \emph{generalized centred simplex gradient} over $\mathbb{Y}=\mathbb{Y}^+\cup\mathbb{Y}^-$, where $\mathbb{Y}^+=\{x^0,x^0+s^i:i\in [p]\}$ and $\mathbb{Y}^-=\{x^0,x^0-s^i:i\in [p]\}$. Detailed explanations and more properties of the generalized centred simplex gradient can be found in \cite{chen2025general,hare2022error}.  In particular, we note that this is a special case of the generalized adapted centred simplex gradient \cite{chen2023adapting,chen2025general} mentioned in Remark~\ref{rem:specialMNgrad=GACSG}.
        
        Moreover, with the structure of $S$ and $T_{1:p}$ defined in Theorem~\ref{thm:QS=MN_diffTi}, the GSH coincides with the \emph{generalized centered simplex Hessian}, and its diagonal elements form the \emph{centered simplex Hessian diagonal}.  Detailed discussions can be found in \cite{jarry2025using}.
    \end{remark}
    
    Once again, similar results exist for the MFN model, with the uniqueness of $\nabla_{\mathrm{MFN}}f(x^0;Y)$ guaranteed by the poisedness of $Y$ for MFN interpolation.
    \begin{theorem}\label{thm:QS=MFN_diffTi}
        Suppose that $S=[e^1\cdots e^p]$ and $T_i=[-e^i]$, $i\in [p]$, where $p\le n$. Then, we have $\nabla_{\mathrm{MFN}}^{2} f(x^0;Y) = \nabla_s^2f(x^0;S;T_{1:p})$. If $Y$ is poised for MFN interpolation, then $\nabla_{\mathrm{MFN}}f(x^0;Y) = \frac{1}{2}\left(\nabla_s f(x^0;S)+\nabla_s f(x^0;-S)\right)$ and $m^{x^0,Y}_{\mathrm{MFN}}(x)$ belongs to the class of QS models of $f$ at $x^0$ over $Y$.
    \end{theorem}

    \begin{remark}
        Comparing Theorems \ref{thm:QS=MN_diffTi} and \ref{thm:QS=MFN_diffTi}, we observe that the MFN, MN, and QS models coincide when $S=[e^1\cdots e^p]$, $T_i=[-e^i]$, and $Y$ is poised for MFN interpolation.  Once again, this is because this particular structure of $Y$ allows Problem \eqref{pro:MN} to be separable in $(\alpha,H)$.
    \end{remark}

\section{Conclusion}\label{sec:conclusion}
    In this paper, we investigated the approximation accuracy and relationships among the MN, MFN, and QS models.  We established fully linear error bounds and derived directional Hessian error bounds for all three models.  In particular, we showed that their Hessian approximations exhibit directional errors ranging from $\mathcal{O}(1)$ to $\mathcal{O}(\Delta_{D^Y})$ and achieve fully quadratic accuracy along sample directions. These results demonstrate that meaningful second-order information can be obtained without fully determined quadratic interpolation.  We also clarified the structural relationships among the three models, identifying conditions under which they coincide.  These insights not only deepen theoretical understanding but also point toward simpler and potentially more efficient strategies for constructing quadratic models in DFO.
    
    A natural direction for future work is to study how different weightings of $\|\alpha\|^2$ and $\|H\|_F^2$ in the objective of Problem \eqref{pro:MN} affect model accuracy and construction. The analysis developed here for the MN and MFN models extends directly to this more general setting.

\bibliographystyle{siam}
\bibliography{references}

@article{hare2024matrix,
  title={A matrix algebra approach to approximate Hessians},
  author={Hare, W. and Jarry-Bolduc, G. and Planiden, C.},
  journal={IMA J. Numer. Anal.},
  volume={44},
  number={4},
  pages={2220--2250},
  year={2024},
  publisher={Oxford University Press}
}

@article{chen2024qfully,
  title={Q-fully quadratic modeling and its application in a random subspace derivative-free method},
  author={Chen, Y. and Hare, W. and Wiebe, A.},
  journal={Comput. Optim. Appl.},
  volume={89},
  number={2},
  pages={317--360},
  year={2024},
  publisher={Springer}
}

@article{chen2023adapting,
    author = {Chen, Y. and Hare, W.},
    title = {Adapting the centred simplex gradient to compensate for misaligned sample points},
    journal = {IMA J. Numer. Anal.},
    year = {2023},
    issn = {0272-4979}
}

@article{chen2025general,
  title={A general framework for floating point error analysis of first-order simplex derivatives},
  author={Chen, Y. and Hare, W. and Wiebe, A.},
  journal={Optim. Methods Softw.},
  year={2025},
  publisher={Taylor \& Francis}
}

@book{conn2009introduction,
  title={Introduction to derivative-free optimization},
  author={Conn, A. R. and Scheinberg, K. and Vicente, L. N.},
  year={2009},
  publisher={SIAM}
}

@incollection{conn1996algorithm,
  title={An algorithm using quadratic interpolation for unconstrained derivative free optimization},
  author={Conn, A. R. and Toint, Ph. L.},
  booktitle={Nonlinear Optimization and Applications},
  pages={27--47},
  year={1996},
  publisher={Springer}
}

@inproceedings{conn1998derivative,
  title={A derivative free optimization algorithm in practice},
  author={Conn, A. R. and Scheinberg, K. and Toint, Ph. L},
  booktitle={7th AIAA/USAF/NASA/ISSMO Symposium on Multidisciplinary Analysis and Optimization},
  year={1998}
}

@article{conn2008geometry,
  title = {Geometry of Sample Sets in Derivative-Free Optimization: Polynomial Regression and Underdetermined Interpolation},
  author = {Conn, A. R. and Scheinberg, K. and Vicente, L. N.},
  year = {2008},
  journal = {IMA J. Numer. Anal.},
  volume = {28},
  number = {4},
  pages = {721--748}
}

@misc{roberts2025introduction,
  title={Introduction to Interpolation-Based Optimization},
  author={Roberts, Lindon},
  note = {\url{https://arxiv.org/abs/2510.04473}},
  year={2025}
}

@article{jarry2025using,
    author = {Jarry-Bolduc, Gabriel and Planiden, Chayne},
    title = {Using generalized simplex methods to approximate derivatives},
    journal = {IMA J. Numer. Anal.},
    year = {2025},
    month = {06}
}

@misc{chen2026relationships,
  title={Relationships between full-space and subspace quadratic interpolation models and simplex derivatives}, 
  author={Chen, Y.},
  year={2026},
  archivePrefix={arXiv},
  primaryClass={math.OC},
  note = {\url{https://arxiv.org/abs/2602.10374}}
}

@article{cartis2023scalable,
  title={Scalable subspace methods for derivative-free nonlinear least-squares optimization},
  author={Cartis, C. and Roberts, L.},
  journal={Math. Program.},
  volume={199},
  number={1-2},
  pages={461--524},
  year={2023},
  publisher={Springer}
}

@article{hare2022error,
  title={Error bounds for overdetermined and underdetermined generalized centred simplex gradients},
  author={Hare, W. and Jarry-Bolduc, G. and Planiden, C.},
  journal={IMA J. Numer. Anal.},
  volume={42},
  number={1},
  pages={744--770},
  year={2022},
  publisher={Oxford University Press}
}

@incollection{Audet2020,
    author={Audet, C.
    and Hare, W.},
    title={Model-based methods in derivative-free nonsmooth optimization},
    bookTitle={Numerical Nonsmooth Optimization},
    year={2020},
    publisher={Springer},
    pages={655--691}
}

@misc{Cao2023,
  title = {The Error in Multivariate Linear Extrapolation with Applications to Derivative-Free Optimization},
  author = {Cao, L. and Wen, Z. and Yuan, Y.-X.},
  note = {\url{https://arxiv.org/abs/2307.00358}},
  year={2023}
}

@book{audet2017derivative,
  title={Derivative-free and blackbox optimization},
  author={Audet, C. and Hare, W.},
  year={2017},
  publisher={Springer}
}

@article{custodio2007using,
  title={Using sampling and simplex derivatives in pattern search methods},
  author={Cust{\'o}dio, A. L. and Vicente, L. N.},
  journal={SIAM J. Optim.},
  volume={18},
  number={2},
  pages={537--555},
  year={2007},
  publisher={SIAM}
}

@article{custodio2008using,
  title={Using simplex gradients of nonsmooth functions in direct search methods},
  author={Cust{\'o}dio, A. L. and Dennis, J. E. and Vicente, L. N.},
  journal={IMA J. Numer. Anal.},
  volume={28},
  number={4},
  pages={770--784},
  year={2008},
  publisher={OUP}
}

@article{regis2015calculus,
  title={The calculus of simplex gradients},
  author={Regis, R. G.},
  journal={Optim. Lett.},
  volume={9},
  number={5},
  pages={845--865},
  year={2015},
  publisher={Springer}
}
\end{document}